\documentclass[10pt,reqno]{amsart}
\usepackage{amscd,amssymb,amsthm}
\usepackage{graphicx}

\usepackage{enumerate}
\theoremstyle{plain}
\newtheorem{theorem}{Theorem}[section]
\newtheorem{corollary}[theorem]{Corollary}
\newtheorem{lemma}[theorem]{Lemma}
\newtheorem{definition}[theorem]{Definition}

\newtheorem{proposition}[theorem]{Proposition}

\theoremstyle{definition}
\newtheorem{remark}[theorem]{Remark}
\newtheorem{application}{Application}

\numberwithin{equation}{section}

\newcommand{\floor}[1]{\left\lfloor{#1}\right\rfloor}
\newcommand{\reel}{\mathbb{R}}
\newcommand{\nat}{\mathbb{N}}
\newcommand{\ent}{\mathbb{Z}}
\newcommand{\comp}{\mathbb{C}}

\newcommand{\ds}{\displaystyle}
\newcommand{\vf}{\varphi}
\newcommand{\eps}{\varepsilon}
\newcommand{\abs}[1]{\left\vert #1\right\vert }
\newcommand{\adh}[1]{\overline{#1} }

\newcommand{\com}[2]{\binom{#1}{#2}}
\newcommand{\bg}{\medskip\goodbreak}
\newcommand{\bn}{\medskip\nobreak}
\newcommand{\impl}{{\quad\Longrightarrow\quad}}
\newcommand{\vers}{{\,\longrightarrow\,}}
\newcommand\egdef{{\,{\buildrel\rm def\over=}\,}}

\DeclareMathOperator{\Log}{Log}

\newcommand{\itemref}[1]{(\textit{\ref{#1}}\,)}

\newcommand*\circled[1]{\text{\raisebox{0.5pt}{\textcircled{\raisebox{-0.9pt}{#1}}}}}

\newenvironment{enumeratea}{\begin{enumerate}%
	[\upshape (a)]}{\end{enumerate}}
\newenvironment{enumeratei}{\begin{enumerate}%
	[\itshape i.]}{\end{enumerate}}
\let\oldtocsection=\tocsection

\let\oldtocsubsection=\tocsubsection

\let\oldtocsubsubsection=\tocsubsubsection

\renewcommand{\tocsection}[2]{\hspace{0em}\oldtocsection{#1}{#2}}
\renewcommand{\tocsubsection}[2]{\hspace{2em}\oldtocsubsection{#1}{#2}}
\renewcommand{\tocsubsubsection}[2]{\hspace{4em}\oldtocsubsubsection{#1}{#2}}

\title[Bernoulli Polynomials and Applications]
{Lecture notes \\
~\\
{\Large Bernoulli Polynomials} \\
{\Large and Applications}}
\author[Omran Kouba]{Omran Kouba$^\dag$}
\address{Department of Mathematics \\
Higher Institute for Applied Sciences and Technology\\
P.O. Box 31983, Damascus, Syria.}
\email{omran\_kouba@hiast.edu.sy}
\thanks{$^\dag$ Department of Mathematics, Higher Institute for Applied Sciences and Technology.}

\begin{document}
\begin{abstract}
In this lecture notes we try to familiarize the audience with the theory of Bernoulli polynomials; we study their properties,
and we give, with proofs and references, some of the most relevant results related to them. 
Several applications to these polynomials are presented, including a unified approach to the asymptotic
 expansion of the error term in many numerical quadrature formul\ae,
and many  new and sharp inequalities, that bound some trigonometric sums.
\end{abstract}
\smallskip\goodbreak

\maketitle
\tableofcontents

\newpage
\section{Introduction}\label{sec1}
\bg
\parindent=0pt
\qquad There are many ways to introduce Bernoulli polynomials and numbers. We opted for the algebraic approach  
relying on the difference operator. But first, let us introduce some notation.

\qquad Let the real vector space of polynomials with real coefficients be denoted by $\reel[X]$. For
a nonnegative integer $n$, let 
$\reel_n[X]$ be the subspace of   $\reel[X]$ consisting of polynomials of degree smaller or equal to $n$.\bg
\qquad If $P$ is a polynomial from $\reel[X]$, we define $\Delta P\egdef P(X+1)-P(X)$, and we denote by $\Delta$
the linear operator, defined on $\reel[X]$, by $P\mapsto\Delta P$.\bg
\begin{lemma}\label{lm11}
The linear operator $\Phi$ defined by
\begin{equation}
\Phi:\reel[X]\vers \reel[X]\times\reel,P\mapsto\left(\Delta P,\int_0^1P(t)\,dt\right)
\end{equation}
is bijective.
\end{lemma}
\begin{proof}
 Consider $P\in\ker \Phi$, then $P\in\ker \Delta$ and $\int_0^1P(t)\,dt=0$. Now, if we consider ${Q(X)=P(X)-P(0)}$, then clearly
we have
\[
Q(X+1)=P(X+1)-P(0)=P(X)-P(0)=Q(X)
\]
This implies by induction that $Q(n)=0$ for every nonnegative integer $n$, so $Q=0$, since it has
infinitely many zeros. Thus, $P(X)=P(0)$, but we have also $\int_0^1P(t)\,dt=0$, so $P(0)=0$, and consequently
$P=0$. This proves that $\Phi$ is injective.\bg
\qquad Clearly, for a nonnegative integer $n$  we have $\deg\Delta(X^{n+1})=n$. Thus 
\[
P\in\reel_{n+1}[X]\impl \Delta P\in \reel_{n}[X]
\]
Therefore,
\[
\forall\,n\in\nat,\quad \Phi(\reel_{n+1}[X])\subset \reel_{n}[X]\times \reel.
\]
But the fact that $\Phi$ is injective implies that
\[
\dim\Phi(\reel_{n+1}[X])=\dim \reel_{n+1}[X]=1+\dim \reel_{n}[X]
=\dim(\reel_{n}[X]\times\reel),
\]
and consequently 
\[
\forall\,n\in\nat,\quad \Phi(\reel_{n+1}[X])= \reel_{n}[X]\times \reel
\]
This,  proves that $\Phi$ is surjective, and the lemma follows.
\end{proof}
\bg
\qquad Let us consider the basis $\mathcal{E}=(e_n)_{n\in\nat}$ of $\reel[X]\times\reel$ defined
by $e_0=(0,1)$ and $e_n=(nX^{n-1},0)$ for $n\in\nat^*$. We can 
define the Bernoulli polynomials, In terms of this basis and of the isomorphism $\Phi$ of Lemma \ref{lm11} as follows:
\bg
\begin{definition}\label{def1} The sequence of \textbf{Bernoulli polynomials} $(B_n)_{n\in\nat}$ is defined by
\[B_n=\Phi^{-1}\left(e_n\right)\qquad\text{for $n\geq0$}.\]
\end{definition}
\bg
\qquad According to Lemma \ref{lm11}, this definition takes a more practical form as follows : 
\bg
\begin{corollary}\label{cor11}
The sequence of Bernoulli polynomials  $(B_n)_{n\in\nat}$ is \textbf{uniquely} defined by the conditions:
\begin{align}
\circled{1}&\quad B_0(X)=1.\notag\\
\circled{2}&\quad \forall\,n\in\nat^*,\quad B_n(X+1)-B_n(X)=nX^{n-1}.\\
\circled{3}&\quad \forall\,n\in\nat^*,\quad \int_0^1B_n(t)\,dt=0.\notag
\end{align}
\end{corollary}
\qquad For instance, it is straightforward to see that 
\[
B_1(X)=X-\frac12,\qquad\text{and }\qquad B_2(X)=X^2-X+\frac16.\]
\bg

\section{Properties of Bernoulli polynomials }\label{sec2}
\bn
\qquad In the next proposition, we summarize some simple properties of Bernoulli {polynomials} :\bg

\begin{proposition}\label{pr21}
The sequence of  Bernoulli polynomials $(B_n)_{n\in\nat}$ satisfies the following properties:
\begin{enumeratei}
\item \label{pr211} For every positive integer $n$ we have $B_n'(X)=nB_{n-1}(X)$.
\item \label{pr212} For every positive integer $n$ we have $B_n(1-X)=(-1)^nB_n(X)$.
\item \label{pr213} For every nonnegative integer $n$ and every positive integer $p$ we have
\[\frac{1}{p}\sum_{k=0}^{p-1}B_n\left(\frac{X+k}{p}\right)
=\frac{1}{p^n}B_n(X).
\]
\end{enumeratei}
\end{proposition}
\begin{proof}
\itemref{pr211}~ Consider the sequence of polynomials $(Q_n)_{n\in\nat}$ defined by
$Q_n={\frac{1}{n+1}B'_{n+1}}$. It is straightforward to see that $Q_0(X)=1$ and for $n\geq1$ :
\[
\Delta Q_n=\frac{1}{n+1}(\Delta B_{n+1})^\prime=\frac{1}{n+1}((n+1)X^n)^\prime=nX^{n-1}
\]
and
\[
\int_0^1Q_n(t)\,dt=\frac{1}{n+1}\int_0^1B'_{n+1}(t)\,dt=\frac{\Delta B_{n+1}(0)}{n+1}=0.
\]
This proves that  the sequence of $(Q_n)_{n\in\nat}$ satisfies the conditions \circled{1}, \circled{2}
and \circled{3} of Corollary \ref{cor11}, and \itemref{pr211} follows because of the unicity assertion.
\bn
 
$(ii)$~ Consider again the sequence $(Q_n)_{n\in\nat}$ defined by
$Q_n(X)={(-1)^nB_n(1-X)}$. Clearly $Q_0(X)=1$ and for $n\geq1$ :
\begin{align*}
\Delta Q_n&=(-1)^n(B_n(-X)-B_n(1-X))=(-1)^{n-1}\Delta B_n(-X)\\
&=(-1)^{n-1}n(-X)^{n-1}=nX^{n-1}.
\end{align*}
Moreover, for  $n\geq1$, $\int_0^1Q_n(t)\,dt=(-1)^n\int_0^1B_{n}(1-t)\,dt=0$. This proves that 
 the sequence of $(Q_n)_{n\in\nat}$ satisfies the conditions \circled{1}, \circled{2}
and \circled{3} of Corollary \ref{cor11}, and \itemref{pr212} follows from the unicity assertion.

\bn 
$(iii)$~ Similarly, consider  the sequence of polynomials $(Q_n)_{n\in\nat}$ defined by
\[
Q_n(X)=p^{n-1}\sum_{k=0}^{p-1}B_n\left(\frac{X+k}{p}\right)
\]
 Clearly $Q_0(X)=1$ and for $n\geq1$ :
\begin{align*}
\Delta Q_n&=p^{n-1}\sum_{k=0}^{p-1}B_n\left(\frac{X+k+1}{p}\right)
-p^{n-1}\sum_{k=0}^{p-1}B_n\left(\frac{X+k}{p}\right)\\
&=p^{n-1}\left(\sum_{k=1}^{p}B_n\left(\frac{X+k}{p}\right)
-\sum_{k=0}^{p-1}B_n\left(\frac{X+k}{p}\right)\right)\\
&=p^{n-1}\left(B_n\left(\frac{X+p}{p}\right)-B_n\left(\frac{X}{p}\right)\right)
=p^{n-1}\Delta B_n\left(\frac{1}{p}X\right)=nX^{n-1}.
\end{align*}
Moreover, for  $n\geq1$,
\begin{align*}
\int_0^1Q_n(t)dt&=p^{n-1}\sum_{k=0}^{p-1}\int_0^1B_n\left(\frac{t+k}{p}\right)\,dt\\
&=p^{n-1}\sum_{k=0}^{p-1}\int_{k/p}^{(k+1)/p}B_n(t)\,dt=p^{n-1}\int_0^1B_n(t)\,dt=0.
\end{align*}
 This proves that 
 the sequence of $(Q_n)_{n\in\nat}$ satisfies the conditions \circled{1}, \circled{2}
and \circled{3} of Corollary \ref{cor11}, and \itemref{pr213} follows by unicity. The proof of Proposition
\ref{pr21} is complete.
\end{proof}
\bg

\begin{definition} \label{def2} The sequence of \textbf{Bernoulli Numbers} $(b_n)_{n\in\nat}$ is defined by
\[b_n=B_n(0)\quad\text{for $n\geq0$}.\]
\end{definition}
\bg
\qquad The following proposition summarizes some properties of Bernoulli numbers.\bg

\begin{proposition}\label{pr22}
The sequence of  Bernoulli numbers $(b_n)_{n\in\nat}$ satisfies the following properties:
\begin{enumeratei}
\item \label{pr221} For every positive integer $n$ we have $b_{2n+1}=0$ and $b_{2n}=B_{2n}(1)$.
\item \label{pr222} For every nonnegative integer $n$ we have $B_n\left(\frac12\right)=(2^{1-n}-1)b_n$.
\item \label{pr223} For every nonnegative integer $n$ we have 
\[B_n(X)=\sum_{k=0}^n\binom{n}{k}b_{n-k}X^k.
\]
\item \label{pr224} For every positive integer $n$ we have 
\[b_n=-\frac{1}{n+1}\sum_{k=0}^{n-1}\binom{n+1}{k}b_{k}.\]
\end{enumeratei}
\end{proposition}
\begin{proof}
$(i)$~Using Proposition \ref{pr21}\itemref{pr211} we have, for $n\geq2$ :
\[
B_n(1)-B_n(0)=\int_0^1B_n'(t)\,dt=n\int_0^1B_{n-1}(t)dt=0
\]
and according to Proposition \ref{pr21}\itemref{pr212} we have $B_n(1)=(-1)^nB_n(0)$ for every $n\geq1$. This
proves that $b_n=0$ if $n$ is an odd integer greater than $2$, and that $B_{2n}(1)=b_{2n}$ for $n\geq1$. This is
\itemref{pr221}.\bg
$(ii)$~ According to Proposition \ref{pr21}\itemref{pr213} with $p=2$ we see that, for every nonnegative
integer $n$ we have
\[
B_n\left(\frac{X}{2}\right)+B_n\left(\frac{X+1}{2}\right)=2^{1-n}B_n(X)
\]
Substituting $X=0$ we obtain \itemref{pr222}.\bg
$(iii)$ ~Consider $n\in\nat$, according Now, we use again Proposition \ref{pr21}\itemref{pr211} to conclude that
\[
B_n^{(k)}=n(n-1)\ldots(n-k+1)B_{n-k},\quad \text{for $0<k\leq n$.}
\]
It follows that
\[
\frac{B_n^{(k)}(Y)}{k!}=\binom{n}{k}B_{n-k}(Y),\quad \text{for $0<k\leq n$}
\]
and using Taylor's formula for polynomials we conclude that
\[
B_n(X+Y)=\sum_{k=0}^n\binom{n}{k}B_{n-k}(Y)X^k
\]
Finally, substituting $Y=0$, we obtain \itemref{pr223}.\bg
$(iv)$ ~Let $n$ be an integer such that $n\geq2$. We have shown that $B_n(1)=b_n$, and using the preceding point
we conclude that $B_n(1)=\sum_{k=0}^n\binom{n}{k}b_{n-k}$, that is $b_n=\sum_{k=0}^n\binom{n}{k}b_{n-k}$
or ${\sum_{k=0}^{n-1}\binom{n}{k}b_{k}=0}$. So
\[
\forall\, n\geq1,\quad \sum_{k=0}^n\binom{n+1}{k}b_k=0
\]
which is equivalent to \itemref{pr224}.
\end{proof}
\bg
\begin{application}\label{rm3}
In Proposition \ref{pr22}~\itemref{pr223}, Bernoulli polynomials are expressed in terms of the canonical basis
$(X^k)_{k\in\nat}$ of $\reel[X]$. In fact, we have proved that
\begin{equation}\label{E:taylor}
B_n(X+Y)=\sum_{k=0}^n\binom{n}{k}B_{n-k}(X)Y^k
\end{equation}
and this can be used to, conversely, express the 
canonical basis $(X^k)_{k\in\nat}$ of $\reel[X]$ in terms of Bernoulli polynomials.\bg
\qquad Indeed, substituting $Y=1$ in \eqref{E:taylor} we obtain
\[
B_{n+1}(X+1)=\sum_{k=0}^{n+1}\binom{n+1}{k}B_{n+1-k}(X),
\]
but according to Corollary \ref{cor11} we have also $B_{n+1}(X+1)=B_{n+1}(X)+(n+1)X^n$. Thus
\[
(n+1)X^n=\sum_{k=1}^{n+1}\binom{n+1}{k}B_{n+1-k}(X)=\sum_{k=0}^{n}\binom{n+1}{k}B_{k}(X)
\]
Finally,
\begin{equation*}\label{E:Xn}
X^n=\frac{1}{n+1}\sum_{k=0}^{n}\binom{n+1}{k}B_{k}(X)
\end{equation*}
\end{application}
\bg
\begin{remark}\label{rm4}
Using the recurrence relation of Proposition \ref{pr22}~\itemref{pr224} we can determine the sequence
of Bernoulli Numbers. In particular, $b_1=-\frac{1}{2}$.  We find  in Table \ref{bn}  the the values of the first six
Bernoulli numbers with even indices. Also we find in Table \ref{BnPol} the list of the first six Bernoulli Polynomials.\bg
\end{remark}
\bg
\begin{table}[h!]
\begin{center}
\begin{tabular}{|c|c|c|c|c|c|c|}
\hline
 $\phantom{\Big|}n\phantom{\Big|}$ &\phantom{-}0~~&\phantom{-}1~&\phantom{-}2~&\phantom{-}3~&\phantom{-}4~&\phantom{-}5~\\ \hline
 $\phantom{\Big|}b_{2n}\phantom{\Big|}$ & \phantom{-}1~~&$\phantom{-}\frac{1}{6}$&$-\frac{1}{30}$&$\phantom{-}\frac{1}{42}$&$-\frac{1}{30}$&$\phantom{-}\frac{5}{66}$\\ \hline
\multicolumn{7}{c}{$\vphantom{\scriptstyle{x}}$}\\
\end{tabular}
\end{center}
\setlength{\abovecaptionskip}{10pt}
\caption{The first Bernoulli Numbers with even indces.}\label{bn}
\end{table}

\bg

\begin{table}[h!]
$\begin{matrix}
\phantom{\Big|}B_0(X)&=1\hfill&\\
\phantom{\Big|}B_1(X)&=X-\frac12\hfill&\\
\phantom{\Big|}B_2(X)&=X^2-X+\frac16\hfill&\\
\phantom{\Big|}B_3(X)&=X^3-\frac32X^2+\frac12X\hfill&=X\left(X-\frac12\right)(X-1)\hfill\\
\phantom{\Big|}B_4(X)&=X^4-2X^3+ X^2-\frac{1}{30}\hfill&\\
\phantom{\Big|}B_5(X)&=X^5-\frac52X^4+\frac53X^3-\frac16X\hfill&=B_3(X)\left(X^2-X-\frac13\right)\hfill\\
\phantom{\Big|}B_6(X)&=X^6-3X^5+\frac52X^4-\frac12X^2+\frac{1}{42}\hfill&\\
\phantom{\Big|}
\end{matrix}$
\caption{The first Bernoulli Polynomials.}\label{BnPol}
\end{table}
\bg
\begin{application}\label{Snapp}
For a nonnegative integer $n$ and a positive integer $m$, we define $S_n(m)$ to be the sum
\[
S_n(m)=1^n+2^n+\cdots+m^n=\sum_{k=1}^m k^n.
\]
Noting that 
\[k^n=\frac{1}{n+1}\left(B_{n+1}(k+1)-B_{n+1}(k)\right)\]
we see that
\[
S_n(m)=\frac{1}{n+1}\left(B_{n+1}(m+1)-B_{n+1}(0)\right).
\]
In Table~\ref{Sn} we have listed the first values of these sums using the results from Table~\ref{BnPol}.
\bg
\begin{table}[ht]
\begin{alignat*}{8}
1\phantom{^2} &{}+{}&2\phantom{^2}&{}+{}&\cdots&{}+{}&m\phantom{^2}&=\frac{m(m+1)}{2}\\
1^2                  &{}+{}&2^2                  &{}+{}&\cdots&{}+{}&m^2                  &=\frac{m(m+1)(2m+1)}{6}\\
1^3                  &{}+{}&2^3                  &{}+{}&\cdots&{}+{}&m^3                  &=\frac{m^2(m+1)^2}{4}\\
1^4                  &{}+{}&2^4                  &{}+{}&\cdots&{}+{}&m^4                  &=\frac{m(m+1)(2m+1)(3m^2+m-1)}{30}
\end{alignat*}
\caption{The sum of consecutive powers.}\label{Sn}
\end{table}
\end{application}

It was on studying these sums that \textit{Jacob Bernoulli} introduced the numbers named after him.

\begin{figure}[h!]
\begin{center}
\includegraphics[width=0.4\textwidth]{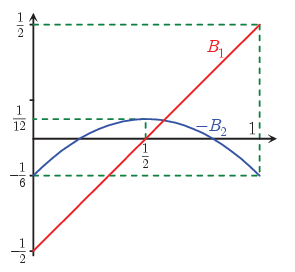}
\caption{The graphs of $B_1$ and $-B_2$ on $[0,1]$.}\label{F1}
\end{center}
\end{figure}

\qquad It is clear that $x\mapsto B_1(x)=x-\frac12$ is negative on $\left(0,\frac12\right)$ and 
positive on $\left(\frac12,1\right)$. It follows that $x\mapsto -B_2(x)=-x^2+x-\frac16$ is increasing
on the interval $\left[0,\frac12\right]$, and decreasing on the interval $\left[\frac12,1\right]$, and since
it has opposit signs at $0$ and $\frac12$ we conclude that $B_2$ vanishes exactly once
on the interval $\left(0,\frac12\right)$ and exactly once on the interval $\left(\frac12,1\right)$.
\bg
\qquad These results can be generalized as follows: 
\bg
\begin{proposition}\label{pr23}
For every positive integer $n$ we have
\begin{description}
\item[$\mathcal{P}_n$] The function $x\mapsto(-1)^nB_{2n}(x)$ is increasing on $\left[0,\frac12\right]$,
and decreasing on $\left[\frac12,1\right]$, and consequently, it vanishes exactly once on each of the intervals
$\left(0,\frac12\right)$ and $\left(\frac12,1\right)$.
\item[$\mathcal{Q}_n$] The function $x\mapsto(-1)^nB_{2n+1}(x)$ is negative on $\left(0,\frac12\right)$,
and positive on $\left(\frac12,1\right)$. Moreover, $0$, $\frac12$ and $1$ are simple zeros of $B_{2n+1}$ in the
interval $[0,1]$.
\end{description}
\end{proposition}
\begin{proof}
We have already proved $\mathcal{P}_1$.\bg
$\mathcal{P}_n\Rightarrow\mathcal{Q}_n$. Let $f(x)=(-1)^nB_{2n+1}(x)$, then we have $f'(x)=(2n+1)(-1)^nB_{2n}(x)$, and
according to $\mathcal{P}_n$, there exists an $\alpha$ in $\left(0,\frac12\right)$ and a $\beta$ in
$\left(\frac12,1\right)$ such that $f'$ is negative on each of the intervals $(0,\alpha)$ and $(\beta,1)$, and  positive
on the interval $(\alpha,\beta)$. Therefore, $f$ has the following table of variations:
\[
\begin{array}{|c|ccccccccc|}
\hline
 x   \phantom{\Big|}&0 & &\alpha& &\frac12& &\beta& &1\\ \hline
f'(x)\phantom{\Big|}&  &-&0& &+& &0&-& \\ \hline
f(x) \phantom{\Big|}&0&\searrow&\smile&\nearrow&0&\nearrow&\frown&\searrow&0\\ \hline
\end{array}
\]
where we used Proposition \ref{pr21}~\itemref{pr212} and Proposition \ref{pr22}~\itemref{pr221} and \itemref{pr222} to conclude 
that \[
B_{2n+1}(0)=B_{2n+1}\left(\frac12\right)=B_{2n+1}(1)=0.
\]
 Moreover, $\mathcal{P}_n$ implies that
$f'$ does not vanish at any of the points $0$, $\frac12$ and $1$. This proves that $0$, $\frac12$ and $1$ are the only
 zeros of $f$ in the interval $[0,1]$ and that they are simple. $\mathcal{Q}_n$ follows immediately.
\bg
$\mathcal{Q}_n\Rightarrow\mathcal{P}_{n+1}$. Let $f(x)=(-1)^{n+1}B_{2n+2}(x)$, then we have $f'(x)=-(2n+2)(-1)^nB_{2n+1}(x)$,
 and
according to $\mathcal{Q}_n$, the derivative $f'$ is positive on $\left(0,\frac12\right)$ and negative
on 
$\left(\frac12,1\right)$. Therefore, $f$ has the following table of variations: 
\[ 
\begin{array}{|c|ccccc|}
\hline
x   \phantom{\Big|}&0& &\frac12& & 1\\ \hline
f'(x)\phantom{\Big|}&0&+&0&-& 0\\ \hline
f(x)\phantom{\Big|}&A&\nearrow&-(1-2^{-1-2n})A&\searrow& A\\ \hline
\end{array}
\]
with $A=(-1)^{n+1}b_{2n+2}$. Clearly $A\ne0$ because $f$ is increasing on $\left(0,\frac12\right)$. Consequently,
$f(0)f\left(\frac12\right)<0$ and $f(1)f\left(\frac12\right)<0$. Thus, $f$ vanishes exactly once on 
each of intervals $\left(0,\frac12\right)$ and $\left(\frac12,1\right)$, and that, it is increasing on
$\left[0,\frac12\right]$ and decreasing $\left[\frac12,1\right]$. This proves $\mathcal{P}_{n+1}$, and achieves 
the proof of the proposition.
\end{proof}
\bn
\begin{figure}[ht]
\begin{center}
\includegraphics[width=0.35\textwidth]{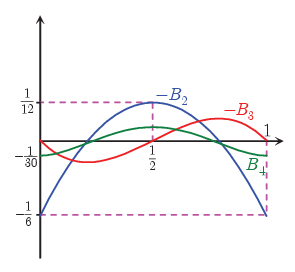}
\caption{Illustration of Proposition \ref{pr23}}\label{F2}
\end{center}
\end{figure} 
\bn
\begin{remark}\label{rm1}
It follows from the preceding proof that $(-1)^{n+1}b_{2n}>0$ for every positive integer $n$.
\end{remark}

\bg
\begin{corollary}\label{cor24}
For every positive integer $n$ we have
\[
\sup_{x\in[0,1]}\abs{B_{2n}(x)}=\abs{b_{2n}}\qquad\text{and}\qquad
\sup_{x\in[0,1]}\abs{B_{2n+1}(x)}\leq\frac{2n+1}{4}\abs{b_{2n}}
\]
\end{corollary}
\begin{proof}
In fact, we conclude from Proposition \ref{pr23} that

\begin{align*}
\sup_{x\in[0,1]}\abs{B_{2n}(x)}&=\max\left(\abs{B_{2n}(0)},\abs{B_{2n}\left(\frac12\right)}\right)\\
&=\max\left(\abs{b_{2n}},\left(1-\frac{1}{2^{2n-1}}\right)\abs{b_{2n}}\right)=\abs{b_{2n}}
\end{align*}
\bg
\qquad In order to show the second inequality we consider several cases:

\begin{itemize}
\item If $x\in\left[0,\frac14\right]$, then we have
\[
B_{2n+1}(x)=B_{2n+1}(x)-B_{2n+1}(0)=\int_0^xB_{2n}(t)\,dt
\]
So
\begin{align*}
\abs{B_{2n+1}(x)}&\leq(2n+1) \int_0^x\abs{B_{2n}(t)}\,dt\\
&\leq  (2n+1)\abs{b_{2n}}x\leq\frac{2n+1}{4}\abs{b_{2n}}
\end{align*}
\item If $x\in\left[\frac14,\frac12\right]$, then we have
\[
B_{2n+1}(x)=B_{2n+1}(x)-B_{2n+1}\left(\frac12\right)=\int_{1/2}^x(2n+1)B_{2n}(t)\,dt
\]
Thus
\begin{align*}
\abs{B_{2n+1}(x)}&\leq(2n+1) \int_x^{1/2}\abs{B_{2n}(t)}\,dt\\
&\leq  (2n+1)\abs{b_{2n}}\left(\frac12-x\right)\leq\frac{2n+1}{4}\abs{b_{2n}}
\end{align*}
\item Finally, when $x\in\left[\frac12,1\right]$, we recall that $B_{2n+1}(x)=-B_{2n+1}(1-x)$ according to Proposition
\ref{pr21} \itemref{pr212}. Thus, in this case we have also
\[
\abs{B_{2n+1}(x)}\leq \sup_{0\leq t\leq\frac12}\abs{B_{2n+1}(t)}\leq\frac{2n+1}{4}\abs{b_{2n}}
\]
\end{itemize}
and the second part of the proposition follows.
\end{proof}
\bg
\begin{remark}\label{rm2}
We will show later in these notes that 
\[ \sup_{x\in[0,1]}\abs{B_{2n+1}(x)}\leq\frac{2n+1}{2\pi}\abs{b_{2n}}\]
which is, asymptotically, the best possible result. That is, we have also
\[
\lim_{n\to\infty}\frac{2\pi}{(2n+1)\abs{b_{2n}}}\cdot\sup_{x\in[0,1]}\abs{B_{2n+1}(x)}=1
\]
\end{remark}
\bg 
\section{Fourier series and Bernoulli polynomials }\label{sec3}
\bn
\qquad Extending periodically the restriction $\left.B_{n}\right|_{[0,1)}$ of the Bernoulli
polynomial $B_n$ to the interval $[0,1)$, we obtain a $1$-periodic piecewise continuous function denoted
 by $\widetilde{B}_n$. In fact, for every real $x$ we have
$\widetilde{B}_n(x)=B_n(\{x\})$ where $\{t\}=t-\floor{t}$ is the fractional part of the real $t$. In Figure \ref{F3}
the graphs of the functions $\widetilde{B}_1$, $\widetilde{B}_2$ and $\widetilde{B}_3$ are depicted.

\begin{figure}[h]
\begin{center}
\includegraphics[width=0.5\textwidth]{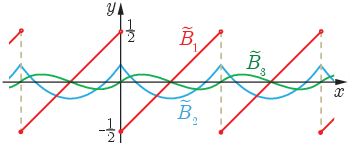}
\caption{The graphs of $\widetilde{B}_1$, $\widetilde{B}_2$ and $\widetilde{B}_3$ }\label{F3}
\end{center}
\end{figure} 

\bg
 In this section we consider the Fourier series expansion of these periodic functions.
\bg
\begin{proposition}\label{pr31}
\begin{enumeratei}
\item \label{pr311} For every $x\in(0,1)$, we have
\[
B_1(x)=-\frac{1}{\pi}\sum_{k=1}^\infty\frac{\sin(2\pi k x)}{k}
\]
\item \label{pr312} For every positive integer $n$, and every $x\in[0,1]$, we have
\begin{align*}
B_{2n}(x)&=(-1)^{n+1}\frac{2(2n)!}{(2\pi)^{2n}}\,\sum_{k=1}^\infty\frac{\cos(2\pi k x)}{k^{2n}}\\
B_{2n+1}(x)&=(-1)^{n+1}\frac{2(2n+1)!}{(2\pi)^{2n+1}}\,\sum_{k=1}^\infty\frac{\sin(2\pi k x)}{k^{2n+1}}
\end{align*}
\end{enumeratei}
\end{proposition}
\begin{proof}
 First, let us consider the case of $\widetilde{B}_1$. It is clear that
\[
C_0(\widetilde{B}_1)=\int_0^1B_1(t)dt=0.
\]
and for $k\ne0$ we have
\begin{align*}
C_k(\widetilde{B}_1)&=\int_0^1B_1(t)e^{-2i\pi kt}dt=\int_0^1\left(t-\frac12\right)e^{-2i\pi kt}dt\\
&=\left.-\left(t-\frac12\right)\frac{e^{-2i\pi kt}}{2i\pi k}\right]_0^1
+\frac{1}{2i\pi k}\int_0^1e^{-2i\pi kt}\,dt=\frac{i}{2\pi k}
\end{align*}
Thus, according to Dirichlet's theorem \cite[Corollary 3.3.9]{gra} we conclude that
for $x\in\reel\setminus\ent$ we have
\[
\widetilde{B}_1(x)=\sum_{k\in\ent\setminus\{0\}}\frac{i}{2\pi k}e^{2i\pi k x}
=-\frac{1}{\pi}\sum_{k=1}^\infty\frac{\sin(2\pi k x)}{k}
\]
and \itemref{pr311} follows.\bg
\qquad Let us now consider the case of $\widetilde{B}_n$, for $n\geq2$. According to Corollary \ref{cor11} we have
\[
C_0(\widetilde{B}_n)=\int_0^1B_n(t)dt=0.
\]
and for $k\ne0$ we find that
\begin{align*}
C_k(\widetilde{B}_n)&=\int_0^1B_n(t)e^{-2i\pi kt}dt\\
&=\left.-B_n(t)\frac{e^{-2i\pi kt}}{2i\pi k}\right]_0^1
+\frac{1}{2i\pi k}\int_0^1B_n^\prime(t)e^{-2i\pi kt}\,dt\\
&=\frac{B_n(0)-B_n(1)}{2i\pi k}+\frac{n}{2i\pi k}\int_0^1B_{n-1}(t)e^{-2i\pi kt}dt\\
&=\frac{n}{2i\pi k}\, C_k(\widetilde{B}_{n-1})
\end{align*}
where we used Corollary \ref{cor11} and Proposition \ref{pr21}~\itemref{pr211}. This allows us to prove by induction
on $n$ that,
\[
\forall\,n\in\nat^*,~\forall\,k\in\ent\setminus\{0\},\qquad
C_k(\widetilde{B}_n)=-\frac{n!}{(2i\pi k)^{n}}
\]
Thus, because of the continuity of $\widetilde{B}_n$ for $n\geq2$, and of the uniform convergence of the Fourier series
of $\widetilde{B}_n$ in this case, we conclude that \cite[Ch.~I, Sec.~3]{katz}, for $x\in \reel$ we have
\[
\widetilde{B}_n(x)=-\sum_{k\in\ent\setminus\{0\}}\frac{n!}{(2i\pi  k)^n}e^{2i\pi k x}=
-\frac{n!}{(2i\pi)^n}\sum_{k=1}^\infty\frac{e^{2i\pi k x}+(-1)^ne^{-2i\pi k x}}{k^n}
\]
and \itemref{pr312} follows by considering separately the cases of $n$ even and $n$ odd.
\end{proof}
\bg
\qquad In particular, we have the following well-known result:
\begin{corollary}\label{cor31}
For $n\geq1$ we have
\begin{align*}
\zeta(2n)&\egdef\sum_{k=1}^\infty\frac{1}{k^{2n}}=(-1)^{n-1}\frac{(2\pi)^{2n}}{2\cdot(2n)!}b_{2n}\\
\eta(2n)&\egdef\sum_{k=1}^\infty\frac{(-1)^{k-1}}{k^{2n}}=(-1)^{n}\frac{(2\pi)^{2n}}{2\cdot(2n)!}B_{2n}\left(\frac12\right)
=(-1)^{n-1}\frac{(2^{2n}-2)\pi^{2n}}{2\cdot(2n)!}b_{2n}
\end{align*}
$($Here $\zeta(\cdot)$ is the well-know ``Riemann Zeta function,''.$)$ 
\end{corollary}

\qquad Using Bessel-Parseval's identity \cite[Ch.~I, Sec.~5]{katz} we obtain the following corollary :\bg
\bg
\begin{corollary}\label{cor32}
If $n$ and $m$ are positive integers then
\[
\int_0^1B_n(x)B_m(x)\,dx=\frac{(-1)^{m-1}}{\binom{n+m}{n}}b_{n+m}
\]
In particular, for $n\in\nat^*$ we have
\[
\int_0^1\left(\frac{B_n(x)}{n!}\right)^2\,dx= (-1)^{n-1} \frac{b_{2n}}{(2n)!}=  \frac{\abs{b_{2n}}}{(2n)!}
\]
\end{corollary}
\begin{proof}
Indeed, if $n+m\equiv1\mod2$ then the change of variables $x\leftarrow1-x$ proves, using Proposition
\ref{pr21}~\itemref{pr212}, that
the considered integral $\int_0^1B_n(x)B_m(x)\,dx $ equals $0$, and the conclusion follows from
Proposition \ref{pr22}~\itemref{pr221}.\bg
So, let us suppose that $n\equiv m\mod2$. In this case, by Bessel-Parseval's identity, we have
\begin{align*}
\int_0^1B_n(x)B_m(x)\,dx&=\sum_{k\in\ent}C_k(\widetilde{B}_n)\adh{C_k(\widetilde{B}_m)}\\
&=(-1)^{m}\frac{n!\,m!}{(2i\pi)^{n+m}}
\sum_{k\in\ent\setminus\{0\}}\frac{1}{k^{n+m}}
\end{align*} 
and the desired conclusion follows by Corollary \ref{cor31}.
\end{proof}
\bg
\begin{application} \textbf{One more formula for Bernoulli polynomials.} 

\qquad Let us consider the polynomial $T_n$ defined by
\[
T_n(X)=\frac{1}{n+1}\sum_{k=0}^{n}B_k(X)B_{n-k}(X).
\]
Clearly, for $n\geq1$, we have
\begin{align*}
(n+1)T'_n(X)&= \sum_{k=1}^{n}k B_{k-1}(X)B_{n-k}(X)+\sum_{k=0}^{n-1}(n-k) B_{k}(X)B_{n-k-1}(X)\\
&= \sum_{k=0}^{n-1}(k+1) B_{k}(X)B_{n-k-1}(X)+\sum_{k=0}^{n-1}(n-k) B_{k}(X)B_{n-k-1}(X)\\
&= (n+1)\sum_{k=0}^{n-1} B_{k}(X)B_{n-1-k}(X)=\frac{n+1}{n}T_{n-1}
\end{align*}
That is $T'_n=nT_{n-1}$.  Now, since $T_n$ is a polynomial of degree $n$ there are $(\lambda_{k}^{(n)})_{0\leq k\leq n}$ such that
$
T_n(X)=\sum_{k=0}^n\lambda_k^{(n)}B_k$, and from $T'_n=nT_{n-1}$ we conclude that
\[
\sum_{k=0}^{n-1}(k+1)\lambda_{k+1}^{(n)}B_{k}=\sum_{k=0}^{n-1}n\lambda_k^{(n-1)}B_k.
\]
Thus $(k+1)\lambda_{k+1}^{(n)}=n\lambda_k^{(n-1)}$ for $0\leq k\leq n-1$. It follows that 
$\lambda_{k}^{(n)}=\com{n}{k} \lambda_0^{(n-k)}$ for $0\leq k\leq n$. So, we have proved
that
\[
\frac{1}{n+1}\sum_{k=0}^{n}B_k(X)B_{n-k}(X)=\sum_{k=0}^n\com{n}{k}\lambda_0^{(k)}B_{n-k}
\]
Integrating on $[0,1]$, and using Corollaries \ref{cor11} and \ref{cor32} , we obtain 
$\lambda_0^{(0)}=1$, $\lambda_0^{(1)}=0$, and for $n\geq 2$:
\begin{align*}
\lambda_0^{(n)}&=\frac{1}{n+1}\sum_{k=0}^{n}\int_0^1B_k(t)B_{n-k}(t)\,dt=
\frac{b_{n}}{n+1}\sum_{k=1}^{n-1}\frac{(-1)^{k-1}}{\binom{n}{k}}\\
&=b_{n}\sum_{k=1}^{n-1}(-1)^{k-1}\frac{k!(n-k)!}{(n+1)!}=-b_n\int_0^1\left(\sum_{k=1}^{n-1}t^{n-k}(t-1)^k\right)\,dt\\
&=
-b_n\int_0^1\big(t^{n+1}-(t-1)^{n+1}-t^n-(t-1)^n\big)\,dt=\frac{1+(-1)^n}{(n+2)(n+1)}b_n.
\end{align*}
That is $\lambda_0^{(2m+1)}=0$ and $\lambda_0^{(2m)}=\frac{1}{(m+1)(2m+1)}b_{2m}$. Therefore, we
have proved that
\[
\frac{1}{n+1}\sum_{k=0}^{n}B_k(X)B_{n-k}(X)= 
\sum_{k=0}^{\floor{n/2}}\com{n}{2k}\frac{b_{2k}}{(k+1)(2k+1)}B_{n-2k}(X).
\]
In particular, taking $n=2m$ and $x=0$ we find, for $m\ne1$, that
\[\sum_{k=0}^{m}b_{2k}b_{2m-2k}=\frac{1}{m+1} \sum_{k=0}^{m}\com{2m+2}{2k+2}b_{2k}b_{2m-2k}.\]
or equivalently, for $m\ne 2$ :
\[\sum_{k=1}^{m}b_{2(k-1)}b_{2(m-k)}=\frac{1}{m} \sum_{k=1}^{m}\com{2m}{2k}b_{2(k-1)}b_{2(m-k)}.\]
 \end{application}
This is  an unusual formula since it combines both convolution and binomial convolution. 

\bg 
\section{Bernoulli polynomials on the unit interval $[0,1]$}\label{sec4}
\bn
\qquad Our first result concerns the sequence of Bernoulli numbers and it follows immediately from
Corollary \ref{cor31}:\bg
\begin{proposition}\label{pr41} The sequence of Bernoulli numbers $(b_{2n})_{n\geq1}$ satisfies the following:

\begin{enumeratei}
\item\label{pr411} For every positive integer $n$, we have
$\ds \abs{b_{2n}}<2\left(1+\frac{3}{2^{2n}}\right)\,\frac{(2n)!}{(2\pi)^{2n}}<4\frac{(2n)!}{(2\pi)^{2n}}$.
\item\label{pr412} Asymptotically, for $n$ in the neighborhood of $+\infty$, we have
$\ds \abs{b_{2n}}\sim_{+\infty}2\frac{(2n)!}{(2\pi)^{2n}}$.
\end{enumeratei}
\end{proposition}
\begin{proof} Noting that, for $t\in[k-1,k]$, we have $k^{-2n}\leq t^{-2n}$ we conclude that
\[
\sum_{k=3}^\infty\frac{1}{k^{2n}} \leq \sum_{k=3}^\infty\int_{k-1}^k\frac{dt}{t^{2n}}
 =\int_2^\infty\frac{dt}{t^{2n}}=\frac{2}{2n-1}\cdot\frac{1}{2^{2n}}
\]
Thus, for every $n\geq1$ we have
\[
1<\zeta(2n)=\sum_{k=1}^\infty\frac{1}{k^{2n}}\leq 1+\frac{1}{2^{2n}}+\frac{2}{2n-1}\cdot\frac{1}{2^{2n}}
\]
or, equivalently
\[
1<\zeta(2n)\leq 1+\frac{2n+1}{2n-1}\cdot\frac{1}{2^{2n}}\leq 1+  \frac{3}{2^{2n}}
\leq 1+ \frac{3}{4}<2
\]
Hence, we have proved that $1<\zeta(2n)<2$ for every $n\geq1$ and that $\lim\limits_{n\to\infty}\zeta(2n)=1$.
This implies the desired conclusion using Corollary \ref{cor31}.
\end{proof}
\bg
\begin{remark}\label{rm5}
Using Stirling's Formula \cite[pp.~257]{abr} we see that
for large $n$ we have
\[
 \abs{b_{2n}}\sim_{+\infty}4\sqrt{\pi n}\left(\frac{n}{e\pi}\right)^{2n}
\]
\end{remark}
\bg
\qquad It is clear, according to Proposition~\ref{pr23} that the function $x\mapsto\abs{B_{2n}(x)}$ attains 
its maximum on the interval $[0,1]$ at $x=0$. Thus, for $n\geq1$ we have
\[
\sup_{x\in[0,1]}\abs{B_{2n}(x)}=\abs{B_{2n}(0)}=\abs{b_{2n}}
\]
\bg
\qquad Determining the maximum of $x\mapsto\abs{B_{2n+1}(x)}$ on the interval $[0,1]$ is more difficult. In this 
regard, we have the following result.
\bg
\begin{proposition}\label{pr42}
For every positive integer $n$ we have

\begin{enumeratei}
\item\label{pr421} $\ds\quad \sup_{x\in[0,1]}\abs{B_{2n+1}(x)}<\frac{2n+1}{2\pi}\abs{b_{2n}}$.
\item\label{pr422} $\ds\quad \abs{B_{2n+1}\left(\frac14\right)}\geq\left(1-\frac{4}{2^{2n}}\right)
\frac{2n+1}{2\pi}\abs{b_{2n}}$.
\end{enumeratei}
\end{proposition}
\begin{proof}
In fact, using Proposition \ref{pr31}~\itemref{pr312} we see that
\[
\sup_{x\in[0,1]}\abs{B_{2n+1}(x)}\leq\frac{2(2n+1)!}{(2\pi)^{2n+1}}\cdot\sum_{k=1}^\infty\frac{1}{k^{2n+1}}
<\frac{2(2n+1)!}{(2\pi)^{2n+1}}\cdot\sum_{k=1}^\infty\frac{1}{k^{2n}}.
\]
Thus, according to Corollary \ref{cor31} we conclude  that
\[
\sup_{x\in[0,1]}\abs{B_{2n+1}(x)}<\frac{2(2n+1)!}{(2\pi)^{2n+1}}\cdot\frac{(2\pi)^{2n}}{2\cdot(2n)!}\abs{b_{2n}}
=\frac{2n+1}{ 2\pi } \abs{b_{2n}}.
\]
which is \itemref{pr421}.\bg
\qquad On the other hand, for $n\in\nat$, using Proposition \ref{pr31}~\itemref{pr312} once more,  we obtain
\[
B_{2n+1}\left(\frac14\right)=(-1)^{n+1}\frac{2(2n+1)!}{(2\pi)^{2n+1}}\sum_{k=0}^\infty\frac{(-1)^k}{(2k+1)^{2n+1}},
\]
but the series above is alternating, so
\[
\sum_{k=0}^\infty\frac{(-1)^k}{(2k+1)^{2n+1}}>1-\frac{1}{3^{2n+1}}
\]
Thus
\[
(-1)^{n+1}B_{2n+1}\left(\frac14\right)>\left(1-\frac{1}{3^{2n+1}}\right) \frac{2(2n+1)!}{(2\pi)^{2n+1}}.
\]
Now, using Proposition \ref{pr41}~\itemref{pr411} we see that
\[
(-1)^{n+1}B_{2n+1}\left(\frac14\right)>\left(\frac{1-3^{-2n-1}}{1+3\cdot 2^{-2n}}\right)\frac{2n+1}{2\pi} \abs{b_{2n}},
\]
and  since
\[
\forall\,n\in\nat^*,\qquad\frac{1-3^{-2n-1}}{1+3\cdot 2^{-2n}}\geq 1-\frac{4}{2^{2n}}
\]
we obtain \itemref{pr422}.
\end{proof}
\bg
\begin{remark}\label{rm41}
Combining Corollary \ref{cor24} and Propositions \ref{pr41} and \ref{pr42} we see that,
for large $n$ we have 
\[
\sup_{x\in[0,1]}\abs{B_n(x)}\sim_{+\infty} \frac{2\cdot n!}{(2\pi)^n}.
\]
\end{remark}
\bg
\qquad Next, we will study the behavior of the unique zero of $B_{2n}$ in the interval $(0,1/2)$.\bg
\begin{proposition}\label{pr43}
For a positive integer $n$, let $\alpha_n$ be the unique zero of $B_{2n}$ that belongs to the interval $(0,1/2)$. Then the sequence
$(\alpha_n)_{n\geq1}$ satisfies the following inequality:
\[
\frac{1}{4}-\frac{1}{\pi\cdot 2^{2n}}<\alpha_n<\alpha_{n+1}<\frac{1}{4}.
\]
\end{proposition}
\begin{proof}
First, note that $ \alpha_1=\frac12-\frac{1}{2\sqrt{3}}$ so 
$\frac14\left(1-\frac{1}{\pi}\right)<\alpha_1<\frac14$. On the other hand, since
$ \alpha_1(1-\alpha_1)=\frac16$ and $  B_4(X)=X^2(1-X)^2-\frac{1}{30}$ we conclude that
\[
B_4(\alpha_1)=-\frac{1}{180}<0\quad\text{and}\quad B_4\left(\frac14\right)=\frac{269}{7680}>0.
\]
This proves the desired inequality for $n=1$.\bg
\qquad Now, let us suppose that $n\geq2$. Using Proposition \ref{pr21}~\itemref{pr213} with $p=4$ and $X=0$ we obtain
\[
B_{2n}(0)+B_{2n}\left(\frac14\right)+B_{2n}\left(\frac12\right)+B_{2n}\left(\frac34\right)=\frac{1}{4^{2n-1}}B_{2n}(0).
\]
But, according to Proposition \ref{pr21}~\itemref{pr212}, 
$B_{2n}\left(\frac14\right)=B_{2n}\left(\frac34\right)$, and using Proposition \ref{pr22}~\itemref{pr222}, we have also
$B_{2n}\left(\frac12\right)=(2^{1-2n}-1)b_{2n}$. Hence
\[
B_{2n}\left(\frac14\right)=\frac{1}{2^{2n}}\left(2^{1-2n}-1\right)b_{2n}=\frac{1}{2^{2n}}B_{2n}\left(\frac12\right).
\]
This proves that $B_{2n}\left(\frac14\right)B_{2n}\left(\frac12\right)>0$ and 
$B_{2n}\left(\frac14\right)B_{2n}(0)<0$. It follows that $0<\alpha_n<\frac14$.\bg
Now, let us define the function $h_n$ by
\[
h_n(x)=\sum_{k=1}^\infty\frac{\cos(2\pi k x)}{k^{2n}}.
\]
Also, let $x_n=\frac{1}{4}-\frac{1}{\pi\cdot 2^{2n}}$. Clearly we have
\begin{align*}
h_n(x_n)&=\cos\left(\frac{\pi}{2}-\frac{2}{2^{2n}}\right)+\sum_{k=2}^\infty\frac{\cos(2\pi k x_n)}{k^{2n}}\\
&\geq \sin\left(\frac{2}{2^{2n}}\right)-\sum_{k=2}^\infty\frac{1}{k^{2n}}
>\sin\left(\frac{2}{2^{2n}}\right)-\frac{2n+1}{2n-1}\cdot\frac{1}{2^{2n}}
\end{align*}
where we used the inequality $\zeta(2n)-1<\frac{2n+1}{2n-1}2^{-2n}$ from Proposition \ref{pr41}. 

Finally, using the inequality $\sin x\geq x-\frac{x^3}{6}$ which is valid for $x\geq0$, 
and recalling that $n\geq 2$
we conclude that
\begin{align*}
h_n(x_n)&>\frac{1}{2^{2n}}\left(2-\frac{2n+1}{2n-1}-\frac{4}{3\cdot 2^{4n}}\right)\\
h_n(x_n)&>\frac{1}{2^{2n}}\left(\frac{2n-3}{2n-1}-\frac{4}{3\cdot 2^{4n}}\right)\geq
\frac{1}{3\cdot 2^{2n}}\left(1-\frac{1}{2^{6}}\right)>0
\end{align*}
This proves, according to Proposition \ref{pr31}~\itemref{pr312}, that 
$B_{2n}(x_n)B_{2n}(0)>0$ and consequently $x_n<\alpha_n$.\bg
\qquad Next, let us show that
$h_{n+1}(\alpha_n)>0$, because this implies, according to Proposition \ref{pr31}~\itemref{pr312}, that
$B_{2n+2}(\alpha_n)B_{2n+2}(0)>0$ and consequently $\alpha_n<\alpha_{n+1}$.\bg
\qquad First, on one hand,  we have
\[
0=h_n(\alpha_n)=\sum_{k=1}^\infty\frac{\cos(2\pi k\alpha_n)}{k^{2n}}
=\cos(2\pi\alpha_n)+\sum_{k=2}^\infty\frac{\cos(2\pi k\alpha_n)}{k^{2n}}
\]
and on the other
\[
h_{n+1}(\alpha_n)=\sum_{k=1}^\infty\frac{\cos(2\pi k\alpha_n)}{k^{2n+2}}
=\cos(2\pi\alpha_n)+\sum_{k=2}^\infty\frac{\cos(2\pi k\alpha_n)}{k^{2n+2}}.
\]
Hence
\begin{align*}
h_{n+1}(\alpha_n)&=-\sum_{k=2}^\infty\left(1-\frac{1}{k^2}\right)\frac{\cos(2\pi k\alpha_n)}{k^{2n}}\\
&=-
\frac{3}{4}\cdot\frac{\cos(4\pi\alpha_n)}{2^{2n}}-\sum_{k=3}^\infty\left(1-\frac{1}{k^2}\right)\frac{\cos(2\pi k\alpha_n)}{k^{2n}}.
\end{align*}
But, from $\frac{1}{4}-\frac{1}{\pi\cdot2^{2n}}<\alpha_n<\frac14$ we conclude that
$\pi-\frac{4}{2^{2n}}<4\pi\alpha_n<\pi$, and consequently
\[
1-\frac{8}{2^{4n}}\leq1-2\sin^2\left(\frac{4}{2^{2n}}\right)
=\cos\left(\frac{4}{2^{2n}}\right)<-\cos(4\pi\alpha_n),
\]
Thus, using the estimate $\sum_{k=3}^\infty\frac{1}{k^{2n}}<\frac{2}{3\cdot 2^{2n}}$ obtained on the occasion of proving
Proposition \ref{pr41}, we get
\begin{align*}
h_{n+1}(\alpha_n)&\geq\frac{3}{2^{2n+2}}-\frac{6}{2^{6n}}-
\sum_{k=3}^\infty\left(1-\frac{1}{k^2}\right)\frac{1}{k^{2n}}\\
&>\frac{3}{2^{2n+2}}-\frac{6}{2^{6n}}-\sum_{k=3}^\infty\frac{1}{k^{2n}}
>\frac{3}{2^{2n+2}}-\frac{8}{2^{6n}}-\frac{2}{3\cdot 2^{2n}}\\
&\geq\frac{1}{2^{2n}}\left(\frac{1}{12}-\frac{8}{2^{4n}}\right)
\geq\frac{1}{2^{2n}}\left(\frac{1}{12}-\frac{1}{32}\right)>0
\end{align*}
This concludes the proof of the desired inequality.
\end{proof}
\bg
\begin{remark} The better inequality: $\alpha_n>\frac14-\frac{1}{2\pi}\, 2^{-2n}$, is proved
in \cite{leh}, but the increasing behaviour of the sequence is not discussed there.
Concerning the \textit{rational} zeros of Bernoulli polynomials, it was proved in \cite{ink} that
the only possible rational zeros for a Bernoulli polynomial are $0$, $\frac12$ and $1$.  
A detailed account of the complex zeros of Bernoulli polynomials can be found in \cite{dil2}.
\end{remark}
\bg
\qquad In the next proposition, we will show how to estimate the ``$L^1$-norm''   $\int_0^1\abs{B_n}$ 
in terms of the ``$L^\infty$-norm'' $\sup_{[0,1]}\abs{B_{n+1}}$.
\bg
\begin{proposition}\label{pr44}  The following two properties hold:
\begin{enumeratei}
\item\label{pr441} For every positive integer $n$, we have
$$\int_0^1\abs{B_n(x)}\,dx<16\frac{n!}{(2\pi)^{n+1}}.$$
\item\label{pr442} Asymptotically, for large $n$, we have
$$\int_0^1\abs{B_n(x)}\,dx\sim_{+\infty}8\frac{n!}{(2\pi)^{n+1}}$$
\end{enumeratei}
\end{proposition}
\bg 
\begin{proof} According to Proposition \ref{pr21} \itemref{pr212} we have $\abs{B_n(x)}=\abs{B_n(1-x)}$. Thus
\[
\int_0^1\abs{B_n(x)}=2\int_0^{1/2}\abs{B_n(x)}\,dx.
\]
So, we consider two cases:\bg
\begin{enumeratea}
\item The case $n=2m$. We have seen (Proposition \ref{pr23}) that ${x\mapsto(-1)^mB_{2m}(x)}$ is increasing on ${[0,1/2]}$ and has a unique zero $\alpha_m$
in this interval. Hence
\begin{align*}
\int_0^{1/2}\abs{B_n(x)}\,dx&=(-1)^m\left(-\int_0^{\alpha_m} B_n(x)\,dx+\int_{\alpha_m}^{1/2} B_n(x)\,dx\right)\\
&=(-1)^m\left(\left.-\frac{B_{n+1}(x)}{n+1}\right]_0^{\alpha_m}+\left.\frac{B_{n+1}(x)}{n+1}\right]_{\alpha_m}^{\frac12}\right)\\
&=2(-1)^{m+1}\frac{B_{2m+1}(\alpha_m)}{n+1}=\frac{2}{n+1}\sup_{x\in[0,1]}\abs{B_{n+1}(x)}.
\end{align*}

Thus, if $n$ is even, we have
\begin{equation*}\label{E:dag1}
\int_0^1\abs{B_n(x)}\,dx=\frac{4}{n+1}\sup_{x\in[0,1]}\abs{B_{n+1}(x)}\tag{\dag}
\end{equation*}
\item The case $n=2m+1$. Using again Proposition \ref{pr23}, we see that the function ${x\mapsto(-1)^{m+1}B_{2m+1}(x)}$ is nonnegative on $[0,1/2]$, thus
\begin{align*}
\int_0^{1/2}\abs{B_n(x)}\,dx&=(-1)^{m+1}\int_0^{1/2} B_n(x)\,dx=\left.(-1)^{m+1}\frac{B_{n+1}(x)}{n+1}\right]_0^{1/2}\\
&=\frac{(-1)^{m+1}}{n+1}\left(B_{2m+2}\left(\frac12\right)-B_{2m+2}(0)\right)\\
&=\frac{(-1)^{m+1}}{n+1}\left(\left(\frac{1}{2^{2m+1}}-1\right)b_{2m+2}-b_{2m+2}\right)\\
&=\left(2-\frac{1}{2^n}\right)\frac{\abs{b_{n+1}}}{n+1}.
\end{align*}
So, according to Corollary \ref{cor24}, if $n$ is odd, we have
\begin{equation*}\label{E:dag2}
\int_0^1\abs{B_n(x)}\,dx=\frac{4-2^{1-n}}{n+1}\sup_{x\in[0,1]}\abs{B_{n+1}(x)}\tag{\ddag}
\end{equation*}
\end{enumeratea}
Thus, combinning \eqref{E:dag1},\  \eqref{E:dag2} and Remark \ref{rm41} we obtain
\[
\int_0^1\abs{B_n(x)}\,dx\sim_{+\infty}\frac{4}{n+1}\sup_{x\in[0,1]}\abs{B_{n+1}(x)}
\sim_{+\infty}\frac{8\cdot n!}{(2\pi)^{n+1}}
\]
Also, using Propositions \ref{pr41} and \ref{pr42} we obtain
\[
\int_0^1\abs{B_n(x)}\,dx<\frac{16\cdot n!}{(2\pi)^{n+1}}
\]
which is the desired conclusion.
\end{proof}

\bg
\section{Asymptotic behavior of Bernoulli polynomials}\label{sec4bis}
\bn
\qquad Proposition~\ref{pr31} shows that, for $x\in[0,1]$, we have
\begin{align*}
\abs{(-1)^{n+1}\frac{(2\pi)^{2n}}{2\cdot(2n)!}B_{2n}(x)-\cos(2\pi x)}&\leq\sum_{k=2}^\infty\frac{1}{k^{2n}}
=\zeta(2n)-1<\frac{3}{2^{2n}}\\
\noalign{\text{and}}
\abs{(-1)^{n+1}\frac{(2\pi)^{2n+1}}{2\cdot(2n+1)!}B_{2n+1}(x)-\sin(2\pi x)}&\leq\sum_{k=2}^\infty\frac{1}{k^{2n+1}}
=\zeta(2n+1)-1<\frac{3}{2^{2n+1}}
\end{align*}
where we used the following simple inequality, valid for $m\geq2$:
\begin{align*}
\zeta(m)-1&<\frac{1}{2^m}+\int_{2}^\infty\frac{dt}{t^m}=\frac{m+1}{m-1}\frac{1}{2^m}<\frac{3}{2^m}
\end{align*}

Thus, the sequence $\left((-1)^{n+1}\frac{(2\pi)^{2n}}{2(2n)!}B_{2n}\right)_{n\geq1}$ converges uniformly
 on the interval $[0,1]$ to $\cos(2\pi\,\cdot)$, and similarly, the sequence 
$\left((-1)^{n+1}\frac{(2\pi)^{2n+1}}{2(2n+1)!}B_{2n+1}\right)_{n\geq1}$ converges uniformly
 on the interval $[0,1]$ to $\sin(2\pi\,\cdot)$.  
In fact, this conclusion is a particular case of a more general result proved by K. Dilcher \cite{dil}. 
\bg
\qquad Let us first introduce some notation.  Let $(T_n)_{n\in\nat}$  be the sequence of polynomials defined by the formula :
\[
T_n(z)=(-1)^{\floor{n/2}}\sum_{k=0}^{\floor{n/2}}(-1)^k\frac{(2\pi z)^{n-2k}}{(n-2k)!}.
\]
So that
\begin{align*}
T_{2n}(z)&=\sum_{k=0}^n(-1)^{k}\frac{(2\pi z)^{2k}}{(2k)!},\\
T_{2n+1}(z)&=\sum_{k=0}^n(-1)^{k}\frac{(2\pi z)^{2k+1}}{(2k+1)!}.
\end{align*}
With this notation we have:
\bg
\begin{proposition} \label{prDil}For every integer $n$, with $n\geq2$, and every complex number $z$ we have
\[
\abs{
        (-1)^{\floor{n/2}}     \frac{(2\pi)^n}{2\cdot n!}B_n\left(z+\frac12\right)-T_n(z)
      }
<\frac{e^{4\pi\abs{z}}}{2^n}
\]
\end{proposition}
\begin{proof}
Note that $B_{2k+1}\left(\tfrac{1}{2}\right)=0$ for every $k\geq0$, and according to Corollary~\ref{cor31} we have
\[
B_{2k}\left(\tfrac{1}{2}\right)=(-1)^k\frac{2\cdot(2k)!}{(2\pi)^{2k}}\,\eta(2k),\quad\text{for $k\geq1$}
\] 
where $\eta(2k)=\sum_{m=1}^\infty(-1)^{m-1}/m^{2k}$.
Thus, using Taylor's expansion we have
\begin{align*}
B_n\left(z+\frac12\right)&=z^n+\sum_{k=1}^n\com{n}{k} B_k\left(\frac12\right)z^{n-k}\\
&=z^n+\sum_{1\leq k\leq n/2}\com{n}{2k} B_{2k}\left(\frac12\right) z^{n-2k}\\
&=z^n+\frac{2\cdot n!}{(2\pi)^n}\sum_{1\leq k\leq n/2}  (-1)^k\eta(2k)\frac{ (2\pi z)^{n-2k}}{(n-2k)!}
\end{align*}
Thus
\[
\frac{(2\pi)^n}{2\cdot n!}B_n\left(z+\frac12\right)-(-1)^{\floor{n/2}}T_n(z)=-\frac{(2\pi z)^n}{2\cdot n!}
+\sum_{1\leq k\leq n/2}  (-1)^k(\eta(2k)-1)\frac{ (2\pi z)^{n-2k}}{(n-2k)!}
\]
and consequently, since $0<1-\eta(2k)<2^{-2k}$, we get
\begin{align*}
\abs{(-1)^{\floor{n/2}}\frac{(2\pi)^n}{2\cdot n!}B_n\left(z+\frac12\right)-T_n(z)}
&\leq \frac{(2\pi \abs{z})^n}{2\cdot n!}+\sum_{1\leq k\leq n/2}  (1-\eta(2k))\frac{ (2\pi \abs{z})^{n-2k}}{(n-2k)!}\\
&\leq\frac{1}{2^n}\left( \frac{(4\pi \abs{z})^n}{2\cdot n!}+\sum_{1\leq k\leq n/2}   \frac{ (4\pi \abs{z})^{n-2k}}{(n-2k)!}\right)\\
&\leq\frac{1}{2^n}\left( \sum_{0\leq k\leq n/2}   \frac{ (4\pi \abs{z})^{n-2k}}{(n-2k)!}\right)\leq\frac{e^{4\pi\abs{z}}}{2^n}
\end{align*}
and the desired inequality follows.
\end{proof}
\qquad Clearly, the sequences of polynomial functions $(T_{2n})_{n\in\nat}$ and $(T_{2n+1})_{n\in\nat}$ 
converge uniformly on every compact subset of $\comp$ to $z\mapsto\cos(2\pi z)$ and $z\mapsto\sin(2\pi z)$ respectively. So,
the next corollary is obtained on replacing $z$ by $z-1/2$ in Proposition \ref{prDil}.
\bg
\begin{corollary}\label{cordil}The following two properties hold:
\begin{enumeratei}
\item The sequence $\left((-1)^{n+1}\frac{(2\pi)^{2n}}{2(2n)!}B_{2n}\right)_{n\geq1}$ converges uniformly
 on every compact subset of $\comp$ to the function $\cos(2\pi\,\cdot)$.
\item The sequence 
$\left((-1)^{n+1}\frac{(2\pi)^{2n+1}}{2(2n+1)!}B_{2n+1}\right)_{n\geq1}$ converges uniformly
on every compact subset of $\comp$  to  the function  $\sin(2\pi\,\cdot)$.  
\end{enumeratei}
\end{corollary}
\bg 
\section{The generating function of Bernoulli polynomials}\label{sec5} 
\bn
\qquad In what follows, we will write $D(a,r)$ to denote the open disk of center $a$ and radius $r$ in the complex plane $\comp$:
\[
D(a,r)=\big\{z\in\comp:\abs{z-a}<r\big\}.
\]
\qquad The next result gives the generating function of the sequence of Bernoulli polynomials. 
\bg
\begin{proposition}\label{pr51}
For every $(z,w)\in\comp\times D(0,2\pi)$, the series $\ds\sum_{n=0}^\infty\frac{B_n(z)}{n!}w^n$ is convergent and
\[
\forall\, (z,w)\in\comp\times D(0,2\pi),\qquad\sum_{n=0}^\infty\frac{B_n(z)}{n!}w^n=\frac{we^{zw}}{e^w-1}
\]
\end{proposition}
\begin{proof}
Using Proposition \ref{pr41} \itemref{pr411}, and the facts that $b_0=1$, $b_1=-1/2$ we see that
\[
\forall\,n\in\nat,\qquad \abs{b_n}\leq 4\,\frac{n!}{(2\pi)^n}.
\]
So, using Proposition \ref{pr22} \itemref{pr223}, we see that  for every nonnegative integer $n$ and complex number $z$ we have:
\begin{align*}
\abs{B_n(z)}&\leq\sum_{k=0}^n\com{n}{k}\abs{b_{n-k}}\abs{z}^k\\
&\leq4\sum_{k=0}^n\com{n}{k}\frac{(n-k)!}{(2\pi)^{n-k}}\abs{z}^k\\
&=4\,\frac{n!}{(2\pi)^{n}}\sum_{k=0}^n\frac{(2\pi\abs{z})^k}{k!}\leq4\,\frac{n!}{(2\pi)^{n}}e^{2\pi\abs{z}}.
\end{align*}
Hence, for every $(z,w)\in\comp\times D(0,2\pi)$ and every nonnegative integer $n$ we have
\[
\frac{\abs{B_n(z)}}{n!}\abs{w}^n\leq 4e^{2\pi\abs{z}}\left(\frac{\abs{w}}{2\pi}\right)^n
\]
This implies the convergence of the series $\sum_{n=0}^\infty\frac{B_n(z)}{n!}w^n$. Therefore, we can define
\[
F:\comp\times D(0,2\pi)\vers\comp, F(z,w)=\sum_{n=0}^\infty\frac{B_n(z)}{n!}w^n
\]
Moreover, for $w\in D(0,2\pi)$, the normal convergence of the series $\sum_{n=0}^\infty\frac{B_n(\cdot)}{n!}w^n$ on
every compact subset of $\comp$, implies, using Proposition \ref{pr21} \itemref{pr211}, the normal convergence
 of the series $\sum_{n=0}^\infty\frac{B^\prime_n(\cdot)}{n!}w^n$ on
every compact subset of $\comp$. So, the function $F(\cdot,w)$ has a derivative on $\comp$ and
\begin{align*}
\frac{\partial F}{\partial z}(z,w)&=\sum_{n=0}^\infty\frac{B^\prime_n(z)}{n!}w^n\\
&=\sum_{n=1}^\infty\frac{nB_{n-1}(z)}{n!}w^n=wF(z,w)
\end{align*}
Thus, there exists a function $f:D(0,2\pi)\vers\comp$ such that $F(z,w)=e^{zw}f(w)$ for every $(z,w)$ in 
${\comp\times D(0,2\pi)}$.
Now, since the series $\sum_{n=0}^\infty\frac{B_n(\cdot)}{n!}w^n$ is normally convergent on the compact set $[0,1]$, and using
Corollary \ref{cor11}, we obtain
\[
\int_0^1F(t,w)\,dt=\sum_{n=0}^\infty\frac{w^n}{n!}\int_0^1B_n(t)\,dt=1
\]
But, on the other hand, we have
\[
\int_0^1F(t,w)\,dt=f(w)\int_0^1e^{tw}\,dt=f(w)\frac{e^w-1}{w}
\]
Hence, $f(w)=w/(e^w-1)$, and consequently $F(z,w)=we^{zw}/(e^w-1)$, which is the desired conclusion.
\end{proof}
\bg
\qquad The above result allows us to find the power series expansion of some well-known functions.
\bg
\begin{proposition}\label{pr52}
The functions $z\mapsto z\cot z$, $z\mapsto \tan z$ and  $z\mapsto z/\sin z$ have the following power series expansions in the
neighbourhood of zero:
\begin{enumeratei}
\item\label{pr521} $\ds\forall\, z\in D(0,\pi),\phantom{\left(0,\frac{\pi}{2}\right)} z\cot z=\sum_{n=0}^\infty\frac{2^{2n}(-1)^nb_{2n}}{(2n)!}z^{2n}$.
\item\label{pr522} $\ds\forall\, z\in D\left(0,\frac{\pi}{2}\right),
\phantom{(0,\pi)} \tan z=\sum_{n=1}^\infty\frac{2^{2n}(2^{2n}-1)(-1)^{n+1}b_{2n}}{(2n)!}z^{2n-1}$.
\item\label{pr523} $\ds\forall\, z\in D(0,\pi),\phantom{\left(0,\frac{\pi}{2}\right)\,\,\,} \frac{z}{\sin z}=\sum_{n=0}^\infty\frac{(2^{2n}-2)(-1)^{n+1}b_{2n}}{(2n)!}z^{2n}$.
\end{enumeratei}
\end{proposition}
\begin{proof} Indeed, choosing $z=0$ in Proposition \ref{pr51} and using Proposition \ref{pr22} \itemref{pr221} we obtain
\[
\forall\, w\in D(0,2\pi),\qquad-\frac{1}{2}w+\sum_{n=0}^\infty\frac{b_{2n}}{(2n)!}w^{2n}=\frac{w}{e^w-1}.
\]
Thus
\[
\forall\, w\in D(0,2\pi),\qquad\sum_{n=0}^\infty\frac{2b_{2n}}{(2n)!}w^{2n}=w\left(\frac{e^w+1}{e^w-1}\right).
\]
Substituting $w=2iz$ we obtain
\[
\forall\, z\in D(0,\pi),\qquad\sum_{n=0}^\infty\frac{2^{2n}(-1)^nb_{2n}}{(2n)!}z^{2n}
=iz\left(\frac{e^{2iz}+1}{e^{2iz}-1}\right)=z\cot z.
\]
This proves \itemref{pr521}.

On the other hand. Noting that
\begin{align*}
\tan z&=\cot z-2\cot(2z)\\
\frac{1}{\sin z}&=\cot\left(\frac{z}{2}\right)-\cot z 
\end{align*}
we obtain \itemref{pr522} and \itemref{pr523}.
\end{proof}
 
\begin{remark}\label{rm51}
Recalling Corollary \ref{cor31} we see that for $z\in D(0,1)$ we have
\begin{equation}\label{E:cots}
\pi z\cot(\pi z)=1-2\sum_{n=1}^\infty\zeta(2n)z^{2n}
\end{equation}
Interchanging the signs of summation we find that
\begin{align*}
\pi z\cot(\pi z)&=1-2\sum_{n=1}^\infty\sum_{k=1}^\infty\left(\frac{z}{k}\right)^{2n}
=1-2\sum_{k=1}^\infty\frac{z^2}{k^2-z^2}\\
&=1+z\sum_{k=1}^\infty\left(\frac{1}{z-k}+\frac{1}{z+k}\right)
\end{align*}
This yields the following simple fraction expansion of the cotangent function:
\begin{equation}\label{E:cotf}
\pi\cot(\pi z)=\frac{1}{z}+\sum_{k=1}^\infty\left(\frac{1}{z-k}+\frac{1}{z+k}\right)
=\lim_{n\to\infty}\sum_{k=-n}^{n}\frac{1}{z-k}
\end{equation}
Note that we have proved this for $z\in D(0,1)$ but the result is valid for every $z\in\comp\setminus\ent$ using analytic continuation
 \cite[Chap.~8, \S~1]{ahl}. Similarly,
for $z\in D(0,1)$ we have
\begin{equation}\label{E:secs}
\frac{\pi z}{\sin(\pi z)}=1-2\sum_{n=1}^\infty\eta(2n)z^{2n}
\end{equation}
Interchanging the signs of summation we find that
\begin{align*}
\frac{\pi z}{\sin(\pi z)}&=1-2\sum_{n=1}^\infty\sum_{k=1}^\infty(-1)^{k-1}\left(\frac{z}{k}\right)^{2n}
=1-2\sum_{k=1}^\infty(-1)^{k-1}\frac{z^2}{k^2-z^2}\\
&=1+z\sum_{k=1}^\infty(-1)^{k}\left(\frac{1}{z-k}+\frac{1}{z+k}\right)
\end{align*}
This yields the following simple fraction expansion of the cosecant function:
\begin{equation}\label{E:secf}
\frac{\pi }{\sin(\pi z)}=\frac{1}{z}+\sum_{k=1}^\infty(-1)^{k}\left(\frac{1}{z-k}+\frac{1}{z+k}\right)
=\lim_{n\to\infty}\sum_{k=-n}^n\frac{(-1)^k}{z-k}
\end{equation}
This is also valid for every $z\in\comp\setminus\ent$ using analytic continuation.
\end{remark}
\bg
\begin{application} Using the power series expansion of $z\mapsto \tan z$ obtained in the previous result, we see that for
every positive integer $n$ we have
\[
\tan^{(2n-1)}(0)=\frac{2^{2n}(2^{2n}-1)}{2n}(-1)^{n+1}b_{2n}
\]
So, let us define $a_n=\tan^{(n)}(0)$. We note that $a_{2n}=0$ for every $n\geq0$ since ``$\tan$'' is an odd function. 
If we use $\tan'=1+\tan^2$  and the Leibniz formula, we obtain,
\[
\tan^{(2n+1)}(z)=\left(1+\tan^2z\right)^{(2n)}=\sum_{k=0}^{2n}\com{2n}{k}\tan^{(k)}(z)\tan^{(2n-k)}(z)
\]
for every positive integer $n$. Thus
\[
\forall\,n\geq1,\qquad a_{2n+1}=\sum_{k=0}^{n-1}\com{2n}{2k+1}a_{2k+1}a_{2(n-k)-1}
\]
But, $a_1=1$ and the above formula shows inductively that $a_{2n+1}$ is an \textit{integer} for every $n$. This proves that
\[
\forall\,n\geq1,\qquad \frac{2^{2n}(2^{2n}-1)}{2n}b_{2n}\in\ent
\]
and considering the separate case of $b_1$ we see that
\begin{equation*}\label{ap3}
\forall\,n\geq1,\qquad \frac{2^{n}(2^{n}-1)}{n}b_{n}\in\ent.
\end{equation*}
This result is to be compared with Corollary~\ref{von5}.
\end{application}
\bg
\begin{application} Let $f(z)=z\cot z$. Since $f(z)-zf'(z)=z^2+f^2(z)$ we conclude from Proposition \ref{pr52} \itemref{pr521} that
\[
\sum_{n=0}^\infty\frac{2^{2n}(1-2n)(-1)^nb_{2n}}{(2n)!}z^{2n}=z^2+\sum_{n=0}^\infty\frac{2^{2n}(-1)^n}{(2n)!}
\left(\sum_{k=0}^{n}\com{2n}{2k}b_{2k}b_{2n-2k}\right)\,z^{2n}
\]
Comparing the coefficients of $z^{2n}$ we see that the sequence $(b_{2n})_{n\geq1}$ can be defined recursively by the formula
\begin{equation*}\label{ap4}
b_2=\frac16,\qquad\forall\,n\geq2,\qquad b_{2n}=-\frac{1}{2n+1}\sum_{k=1}^{n-1}\com{2n}{2k}b_{2k}b_{2n-2k}.
\end{equation*}
\end{application}

\begin{application} \textbf{A multiplication formula for Bernoulli polynomials.} Consider an integer
$q$, with $q\geq2$, clearly we have
\begin{align*}
\frac{e^{qw}-1}{e^w-1}&=\sum_{k=0}^{q-1}e^{kw}=1+\sum_{n=0}^\infty\left(\sum_{k=1}^{q-1}k^n\right)\frac{w^n}{n!},\\
&=1+\sum_{n=0}^\infty S_n(q-1)\frac{w^n}{n!}=1+
\sum_{n=0}^\infty\frac{B_{n+1}(q)-B_{n+1}(0)}{n+1}\cdot\frac{w^n}{n!},\\
&=q+\sum_{n=1}^\infty \frac{B_{n+1}(q)-b_{n+1}}{n+1}\cdot\frac{w^{n}}{n!}.
\end{align*}
Where we used the notation of Application \ref{Snapp}. Noting the identity
\[
q\cdot\frac{we^{(qz)w}}{e^w-1}=  \frac{(qw)e^{z(qw)}}{e^{qw}-1}\cdot\frac{e^{qw}-1}{e^w-1}
\] we conclude that
\begin{align*}
\sum_{n=0}^\infty qB_n(qz)\frac{w^n}{n!}&=\left(\sum_{n=0}^\infty q^nB_n(z)\frac{w^n}{n!}
\right)\left(q+\sum_{n=1}^\infty \frac{B_{n+1}(q)-b_{n+1}}{n+1}\cdot\frac{w^{n}}{n!}\right)\\
&=\sum_{n=0}^\infty G_n(q,z)\frac{w^n}{n !}
\end{align*} 
where 
\[
G_n(q,z)=q^{n+1}B_n(z)+\frac{1}{n+1}\sum_{j=0}^{n-1}q^{j}\com{n+1}{j}B_{j}(z)
(B_{n+1-j}(q)-b_{n+1-j}).
\]
But, because a power series expansion is unique, we have $qB_n(qz)=G_n(q,z)$ for every $n$. Now, fix $z$ in $\comp$ and consider
the polynomial
\[
Q(X)= B_n(zX)-\left(X^{n}B_n(z)+\frac{1}{n+1}\sum_{j=0}^{n-1}\com{n+1}{j}B_{j}(z)
(B_{n+1-j}(X)-b_{n+1-j})X^{j-1}\right)
\] 
(Note that $X|(B_{n+1-j}(X)-b_{n+1-j})$.) Clearly $\deg Q\leq n$, and $Q$  has infinitely many zeros, (namely, every integer $q$ greater than $1$.) 
Thus $Q(X)=0$ and we have proved the following ``Multiplication Formula'', valid for every complex numbers $z$ and $w$:
\begin{equation}\label{mul}
B_n(zw)=w^{n}B_n(z)+\frac{1}{n+1}\sum_{j=0}^{n-1}\com{n+1}{j}B_{j}(z)w^{j-1}(B_{n+1-j}(w)-b_{n+1-j})
\end{equation}
For example, taking $w=2$ and $z=0$ we obtain, the following recurrence 
\[
b_n=\frac{1}{2(1-2^n)} \sum_{j=0}^{n-1}2^{j}\com{n}{j}b_{j} 
\]
since $B_{n+1-j}(2)-B_{n+1-j}(0)=n+1-j$ for $0\leq j<n$ according to Corollary~\ref{cor11}. This recurrence was obtained in \cite{nam},
and was generalized in \cite{dee}. All these generalizations follow from \eqref{mul}.
\end{application}
\bg 
\begin{application} \textbf{More formul\ae~ for Bernoulli numbers.} The function
$w\mapsto \frac{e^w-1}{w}$ is entire, and has a power series expansion, that converges in the whole complex plane.
Let $\rho$ be defined by
\[\frac{1}{\rho}=\sup_{w\in D(0,1)}\abs{\frac{e^w-1}{w}}>1.
\]
 Now, for the disk $w\in D(0,\rho)$ we have $\abs{e^w-1}<1$ and consequently
\[
w=\Log(1-(1-e^w))=-\sum_{n=1}^\infty\frac{1}{n}(1-e^w)^n.
\]
Thus, for every $m\geq 1$ we have
\[
\frac{w}{e^w-1}=\sum_{n=0}^{m}\frac{(1-e^w)^{n}}{n+1} +g_m(w)
\]
where $g_m(w)=\sum_{n=m+1}^\infty\frac{1}{n+1}(1-e^w)^n$. Clearly, $w=0$ is a zero
of $g_m$ of order  greater than $m$. Thus $ g_m^{(m)}(0)=0$. But, using Proposition \ref{pr51} the Bernoulli number $b_m$ is the 
$m^{\text{th}} $ derivative of $w\mapsto \frac{w}{e^w-1}$ at $0$ so
\[
b_m=\sum_{n=0}^m\frac{1}{n+1}\left.\big((1-e^w)^{n}\big)^{(m)}\right]_{w=0}
\]
But  $(1-e^w)^{n}=\sum_{k=0}^n\com{n}{k}(-1)^{k}e^{kw}$. Thus,
\begin{equation}\label{gould}
b_m=\sum_{n=0}^m\frac{1}{n+1}\left(\sum_{k=0}^n\com{n}{k}(-1)^{k}k^m\right),\quad\text{for $m\geq1$.}
\end{equation}
This is  quite an old formula for Bernoulli numbers (see \cite{gld} and the references therein.) 
Noting  that
\begin{align*}
\sum_{n=1}^m\frac{(1-x)^n}{n}&=\int_0^{1-x}\left(\sum_{n=1}^mt^{n-1}\right)\,dt
=\int_0^{1-x}\frac{t^m-1}{t-1}\,dt\\
&=\int_{1}^{x}\frac{(1-u)^m-1}{u}\,du=\int_1^{x}\left(\sum_{n=1}^m\com{m}{n}(-1)^nu^{n-1}\right)\,du\\
&=\sum_{n=1}^m\com{m}{n}(-1)^n\frac{x^{n}-1}{n},
\end{align*}
we can rearrange our previous calculation, as follows
\[
\sum_{n=1}^m\frac{(1-e^w)^{n-1}}{n} =\sum_{n=1}^m\com{m}{n}\frac{(-1)^n}{n}\frac{e^{nw}-1}{e^w-1} 
 =\sum_{n=1}^m\com{m}{n}\frac{(-1)^n}{n}\left(\sum_{k=0}^{n-1}e^{kw}\right).
\]
Hence, taking  as before  the $m^{\text{th}} $ derivative at $0$, another formula is obtained \cite{ber}:
\begin{equation}
b_m=\sum_{n=1}^m\com{m}{n}\frac{(-1)^n}{n}\left(\sum_{k=1}^{n-1}k^m\right),\quad\text{for $m\geq1$.}
\end{equation}
\end{application}
\bg
\section{The von Staudt-Clausen theorem}\label{sec6bis}
\bn
\qquad In this section we give the proof of a famous theorem that determines the fractional part of a Bernoulli number. First, let us
introduce some notation, the reader is invited to take a look at \cite[Chapter~15]{Ire}, and the references therein, for
a deeper insight on the role played by Bernoulli numbers in Number Theory.

\qquad Let us denote by $\mathfrak{A}$ the set of functions $f$ that are analytic in the neighborhood
of $0$ and such that $f^{(n)}(0)$ is an integer for every nonnegative integer $n$. Let $f$ and $g$  be two members of $\mathfrak{A}$,
and let $m$ be a positive integer.  
We will write $f\equiv g\pmod{m}$  if $f^{(n)}(0)\equiv g^{(n)}(0)\pmod{m}$ for every nonnegative integer $n$. Finally,
for two functions $f$ and $g$ that are analytic functions in the  neighborhood
of $0$, we write $f\equiv g\pmod{\mathfrak{A}}$ if $f-g\in\mathfrak{A}$.
\bg 
\begin{lemma}\label{von1} The following properties hold:
\begin{enumeratei}
\item \label{von11}If $f$ belongs to $\mathfrak{A}$, then both $f'$ and $z\mapsto\int_0^zf(t)\,dt$ belong to  $\mathfrak{A}$.
\item \label{von12}If $f$ and $g$ belong to  $\mathfrak{A}$, then $fg$ belongs also to  $\mathfrak{A}$.
\item \label{von13}If $f$ belongs to $\mathfrak{A}$ and $f(0)=0$, then $\frac{1}{m!}f^m$ belongs also to  $\mathfrak{A}$ for every
positive integer $m$.
\end{enumeratei}
\end{lemma}
\begin{proof} Consider $f\in\mathfrak{A}$. There is a sequence of integers $(a_n)_{n\in\nat}$
such that
$f(z)=\sum_{n=0}^\infty\frac{a_n}{n!}z^n$ in a neighbourhood of $0$. But then
\[
f'(z)=\sum_{n=0}^\infty\frac{a_{n+1}}{n!}z^n,\quad\text{and}\quad
\int_0^zf(t)dt=\sum_{n=1}^\infty\frac{a_{n-1}}{n!}z^{n}.
\]
and consequently both $f'$ and $z\mapsto\int_0^zf(t)\,dt$ belong to $\mathfrak{A}$. This proves \itemref{von11}.

Property \itemref{von12} follows from Leibniz formula.\bg
Property \itemref{von13} is proved by mathematical induction. It is true for $m=1$, and if we suppose that
$\frac{1}{(m-1)!}f^{m-1}$ belongs to $\mathfrak{A}$, then according to \itemref{von11} and \itemref{von12} the function
$\frac{1}{(m-1)!}f^{m-1}f'$ belongs also to $\mathfrak{A}$, and using \itemref{von11} once more we conclude that 
\[z\mapsto \frac{1}{m!}f^m(z)=\frac{1}{(m-1)!}\int_0^zf^{m-1}(t)f'(t)dt
\]
also belongs  $\mathfrak{A}$. This achieves the proof of the lemma.
\end{proof}
\bg
\begin{lemma}\label{von2}
If $m$ is a composite positive number such that $m>4$ then 
\[(m-1)!\equiv0\pmod{m}.\]
\end{lemma}
\begin{proof}
Let $p$ be the smallest prime that divides $m$, and let $q=m/p$. Since $m$ is composite we conclude that
$q\geq p$  so, there are two cases:
\begin{itemize}
\item $q>p$. In this case $1<p<q<m$ and consequently $m=pq$ divides $(m-1)!$.
\item $q=p$. That is $m=p^2$, but $m>4$, implies that $p>2$ and consequently $1<p<2p<m$.
It follows that $2m=p\times(2p)$ divides $(m-1)!$.
\end{itemize}
and the lemma follows.\end{proof}
\bg
\begin{proposition} \label{von3}The function $g$ defined by $g(z)=e^z-1$ belongs to $\mathfrak{A}$ and it satisfies the following properties:
\begin{enumeratei}
\item \label{von31}If $m$ is composite and greater than $4$ then $g^{m-1}\equiv0\pmod{m}$. 
\item \label{von32}If $m=4$ then $\ds g^{m-1}(z)\equiv2\sum_{k=0}^\infty\frac{z^{2k+1}}{(2k+1)!}\pmod{m}$.
\item \label{von33}If $m$ is prime then $\ds g^{m-1}(z)\equiv-\sum_{k=1}^\infty\frac{z^{k(m-1)}}{(km-k)!}\pmod{m}$.
\end{enumeratei}
\end{proposition}
\begin{proof} The fact that $g\in\mathfrak{A}$ is immediate.

\qquad Suppose that $m$ is a composite integer greater than $4$. Using Lemma~\ref{von1} \itemref{von13} we see that ${\frac{g^{m-1}}{(m-1)!}\in\mathfrak{A}}$, and \itemref{von31} follows from Lemma~\ref{von2}.

\qquad Consider the case $m=4$. Noting that $g^3(z)=e^{3z}-3e^{2z}+3e^z-1$
 we conclude that $g^3(z)=\sum_{n=3}\frac{a_n}{n!}z^n$ with
\[
a_n=3^n-3\cdot 2^n+3.
\]
But, for $n\geq3$ we have $a_n\equiv(-1)^n-1\pmod{4}$ so $a_{2k}\equiv0\pmod{4}$ and $a_{2k+1}\equiv2\pmod{4}$. This proves \itemref{von32}

\qquad Finally, suppose that $m$ is a prime. Here  
\[g^{m-1}(z)=\sum_{k=0}^{m-1}\com{m-1}{k}(-1)^{m-k-1}e^{kz},\]
and consequently $g^{m-1}(z)=\sum_{n=m-1}\frac{b_n}{n!}z^n$ with 
\[
b_0=b_1=\ldots=b_{m-2}=0,~b_{m-1}=(m-1)!,
\]
and for $n\geq m-1$
\[
b_n=\sum_{k=1}^{m-1}\com{m-1}{k}(-1)^{m-k-1}k^n
\]
But according to  Fermat's Little Theorem \cite[Theorem~71]{har} we have $k^{m-1}\equiv1\pmod{m}$ for
$1\leq k\leq m-1$, and consequently
$b_{n+m-1}\equiv b_n\pmod{m}$ for every $n$. Thus,
$b_n\equiv 0\pmod{m}$ if $n$ is not a multiple of $m-1$, and if $n$ is a multiple of $m-1$ then
\[
b_n\equiv(m-1)!\equiv-1\pmod{m}\]
where the last congruence follows from Wilson's Theorem \cite[Theorem~80]{har}, and \itemref{von33} follows.
\end{proof}
\begin{theorem}[von Staudt-Clausen Theorem]\label{von4}
For a given positive integer $n$, let the set of primes $p$ such that $p-1$ divides $2n$ be denoted by
$\mathfrak{p}_n$. Then 
\[
b_{2n}+\sum_{p\in\mathfrak{p}_n}\frac{1}{p}\in\ent.
\]
\end{theorem}
\begin{proof}
Indeed, consider the function $g$ of Proposition~\ref{von3}. Note that
\[
z=\Log(1+g(z))=\sum_{m=1}^\infty\frac{(-1)^{m-1}}{m}g^m(z)
\]
Thus
\begin{align*}
\frac{z}{e^z-1}&=\sum_{m=1}^\infty\frac{(-1)^{m-1}}{m}g^{m-1}(z)=1-\frac{g(z)}{2}-\frac{g^3(z)}{4}+
\sum_{\substack{p>2\\ p\text{ prime}}}\frac{g^{p-1}(z)}{p} \pmod{\mathfrak{A}}\\
&=1-\frac{1}{2}\sum_{k=1}^\infty\frac{z^{2k+1}}{(2k+1)!}-
\sum_{\substack{p\geq2\\ p\text{ prime}}}\frac{1}{p}
\sum_{k=1}^\infty\frac{z^{k(p-1)}}{(kp-k)!} \pmod{\mathfrak{A}}
\end{align*}
But, since $z/(e^z-1)=\sum_{n=0}^\infty\frac{b_n}{n!}z^n$ according to Proposition~\ref{pr51}, the desired
conclusion follows from the above equality, on comparing the coefficients of $z^{2n}$.
\end{proof}

\qquad For instance,   $\mathfrak{p}_1=\{2,3\}$ and
$b_2+\frac12+\frac13=1$. Also, $\mathfrak{p}_2=\{2,3,5\}$, and $b_4+ \frac12+\frac13+\frac{1}{5}=1$.
Generally, for a positive integer $n$,
we have $\{2,3\}\subset\mathfrak{p}_n$ and consequently the denominator of $b_{2n}$ is always a multiple of $6$.\bg
\begin{corollary}\label{von5}
For every positive integer $m$ and every  nonnegative integer $k$ the quantity $m(m^k-1)b_k$ is an integer.
\end{corollary}
\begin{proof} 
We only need to consider the case $k=2n$ for some positive integer $n$ because the other cases are trivial.

\qquad Consider a prime $p$ such that $p-1$ divides $2n$, (\textit{i.e.} $p\in\mathfrak{p}_n$.) Consider also a positive
integer $m$. 
\begin{itemize}
\item If $p \mid m$ then clearly $p\mid m(m^k-1)$.
\item If $p \nmid m$ then, according to  Fermat's Little Theorem \cite[Theorem~71]{har} we have
$m^{p-1}\equiv 1\pmod{p}$, and since $(p-1)\mid k$ we conclude that
$m^k\equiv 1\pmod{p}$. Thus $p\mid m(m^k-1)$ also in this case.
\end{itemize}
It follows that $m(m^k-1)$ is a multiple of every prime  $p\in\mathfrak{p}_n$, and the result follows
according to Proposition~\ref{von4}.
\end{proof}
\begin{corollary}\label{von6}
For every positive integer $n$ there are infinitely many integers $m$ such that $B_{2m}-B_{2n}$  is an integer.
\end{corollary}
\begin{proof} 
Consider $\tau=\text{lcm}(d+1:d\mid (2n) )$; (the least common multiple
of the numbers $d+1$ where $d$ is a divisor of $2n$,) and let $q$ be a \textit{prime} number such that ${q\equiv 1\pmod{\tau}}$.\bg
Now, if $m=nq$ then $\mathfrak{p}_m=\mathfrak{p}_n$. Indeed, if
$p'\in\mathfrak{p}_m$ then $p'-1$ divides $2nq$, so, there are two cases:

\begin{itemize}
\item If $p'-1$ divides $2n$ then clearly $p'\in\mathfrak{p}_n$ since $p'$ is prime.
\item If $p'-1=dq$ for some $d\mid 2n$, then
\[p'=1+dq=(d+1)q+1-q=(d+1)q-\lambda\tau,\quad\text{for some integer $\lambda$},\]
so, $p'$ is a multiple of $d+1$, and since it is a prime, we conclude that $p'=d+1$ which is absurd since $q\ne1$.
\end{itemize}
Thus, we have proved that $\mathfrak{p}_m\subset \mathfrak{p}_n$. But, the inverse inclusion is trivially true,
and $\mathfrak{p}_m=\mathfrak{p}_n$, or equivalently $B_{2m}-B_{2n}$ is an integer.

\qquad Finally, using Dirichlet's Theorem \cite[Chapter~16]{Ire}, we know that there are infinitely many primes $q$ such that
$q\equiv 1\pmod{\tau}$, and the corollary follows.
\end{proof}
 
\bg
\section{The Euler-Maclaurin's formula}\label{sec6}
\bn
\qquad For a function $g$ defined on the interval $[0,1]$ we introduce the notation $\delta g$ to denote the difference
$g(1)-g(0)$. Also we recall the notation $\widetilde{B}_n$ for the $1$-periodic function that coincides with $x\mapsto B_n(x)$ on the interval
$[0,1]$, or equivalently, 
\[
\forall\,x\in\reel,\qquad\widetilde{B}_n(x)=B_n(\{x\})
\]
where $\{t\}=t-\floor{t}$ is the fractional part of $t$.\bg
\begin{proposition}\label{pr61}
Consider a positive integer $m$ and a  function $f:[0,1]\vers\comp$ having a continuous $m^{\text{th}}$ derivative. 
For every $x$ in $[0,1]$ we have
\[
\int_0^1f(t)\,dt-f(x)+\sum_{k=0}^{m-1}\frac{B_{k+1}(x)}{(k+1)!}\,\delta f^{(k)}=\frac{1}{m!}
\int_0^1\widetilde{B}_m(x-t) f^{(m)}(t)\,dt
\]
\end{proposition}
\begin{proof} For an integer $k$ with $0\leq k\leq m$ we define $F_k(x)$ by the formula
\[
F_k(x)=\frac{1}{k!}
\int_0^1\widetilde{B}_k(x-t) f^{(k)}(t)\,dt.
\]
Clearly we have
\[
F_0(x)=\int_0^1\widetilde{B}_0(x-t) f(t)\,du=\int_0^1  f(t)\,dt.
\]
Also, for $0\leq k<m$ and $x\in[0,1]$, we have
\begin{align*}
F_{k+1}(x)&=\int_0^1 \frac{\widetilde{B}_{k+1}(x-t)}{(k+1)!}f^{(k+1)}(t)\,dt\\
&=\int_0^x \frac{B_{k+1}(x-t) }{(k+1)!}\,f^{(k+1)}(t)\,dt+\int_x^1 \frac{B_{k+1}(1+x-t)}{(k+1)!}\,f^{(k+1)}(t)\,dt\\
&=\left.\frac{B_{k+1}(x-t) }{(k+1)!}f^{(k)}(t)\right]_{t=0}^{t=x}+\int_0^x \frac{B_{k}(x-t)}{k!}\,f^{(k)}(t)\,dt+\\
&\phantom{=\,\,}\left.\frac{B_{k+1}(1+x-t) }{(k+1)!}f^{(k)}(t)\right]_{t=x}^{t=1}+\int_x^1 \frac{B_{k}(1+x-t)}{k!}\,f^{(k)}(t)\,dt\\
&=\frac{B_{k+1}(0)-B_{k+1}(1)}{(k+1)!}f^{(k)}(x)+\frac{B_{k+1}(x)}{(k+1)!}\delta f^{(k)}+\\
&\phantom{=\,\,}\int_0^x \frac{\widetilde{B}_{k}(x-t)}{k!}\,f^{(k)}(t)\,dt+\int_x^1 \frac{\widetilde{B}_{k}(x-t)}{k!}\,f^{(k)}(t)\,dt\\
&=\frac{B_{k+1}(0)-B_{k+1}(1)}{(k+1)!}f^{(k)}(x)+\frac{B_{k+1}(x)}{(k+1)!}\delta f^{(k)}+
F_k(x)
\end{align*}
Hence, we have proved that
\begin{align*}
F_0(x)&=\int_0^1f(t)\,dt\\
F_1(x)&=-f(x)+B_1(x)\delta f+F_0(x)\\
F_{k+1}(x)&=\frac{B_{k+1}(x)}{(k+1)!}\delta f^{(k)}+F_k(x)\quad\text{for $1\leq k<m$}
\end{align*}
Adding these equalities as $k$ varies from $0$ to $m-1$ we obtain the desired formula.
\end{proof}
\bg
\qquad The next corollary corresponds to the particular case $x=1$.
\bg
\begin{corollary}\label{cor61}
Consider a positive integer $m$, and a function $f$ that has a  continuous $(2m-1)^{\text{st}}$ derivative on $[0,1]$. If
 $f^{(2m-1)}$ is {\normalfont \text{decreasing}}, then
\[
\int_0^1f(t)\,dt=\frac{f(1)+f(0)}{2}
-\sum_{k=1}^{m-1}\frac{b_{2k}}{(2k)!}\,\delta f^{(2k-1)}+(-1)^{m+1}R_m
\]
with
\[
R_m=\int_0^{1/2}\frac{\abs{B_{2m-1}(t)}}{(2m-1)!}\,\left(f^{(2m-1)}(t)-f^{(2m-1)}(1-t)\right)\,dt
\]
and
\[
0\leq R_m\leq\frac{6}{(2\pi)^{2m}}\left(f^{(2m-1)}(0)-f^{(2m-1)}(1)\right).
\]
\end{corollary}
\begin{proof}
Indeed, choosing $x=1$ in Proposition \ref{pr61} with $2m-1$ for $m$, we obtain
\[
\int_0^1f(t)\,dt-f(1)+\sum_{k=0}^{2m-2}\frac{B_{k+1}(1)}{(k+1)!}\,\delta f^{(k)}=\frac{1}{(2m-1)!}
\int_0^1B_{2m-1}(1-t) f^{(2m-1)}(t)\,dt
\]
Now, using Proposition \ref{pr21} \itemref{pr212}, Proposition \ref{pr22} \itemref{pr221} and the fact that $B_1(1)=1/2$, we see
that
\[
\int_0^1f(t)\,dt-\frac{f(1)+f(0)}{2}
+\sum_{k=1}^{m-1}\frac{b_{2k}}{(2k)!}\,\delta f^{(2k-1)}=-\frac{r_m}{(2m-1)!}
\]
with
\[
r_m=\int_0^1 B_{2m-1}(t) f^{(2m-1)}(t)\,dt
\]
But, 
\begin{align*}
r_{m}&=\int_0^{1/2} B_{2m-1}(t)\, f^{(2m-1)}(t)\,dt+
\int_{1/2}^1 B_{2m-1}(t)\,f^{(2m-1)}(t)\,dt\\
&=\int_0^{1/2}B_{2m-1}(t)\,f^{(2m-1)}(t)\,dt+
\int_0^{1/2}B_{2m-1}(1-t)\,f^{(2m-1)}(1-t)\,dt\\
&=\int_0^{1/2}B_{2m-1}(t)\,\left(f^{(2m-1)}(t)-f^{(2m-1)}(1-t)\right)\,dt
\end{align*}
Now, according to Proposition \ref{pr23}, we know that $(-1)^mB_{2m-1}$ is positive on $(0,1/2)$. Thus
\[
r_m=(-1)^m\int_0^{1/2}\abs{B_{2m-1}(t)}\,\left(f^{(2m-1)}(t)-f^{(2m-1)}(1-t)\right)\,dt,
\]
and the expression of $R_m$ follows.\bg
\qquad In particular, when $f^{(2m-1)}$ is decreasing, the maximum on the interval $[0,1/2]$ of the quantity ${f^{(2m-1)}(t)-f^{(2m-1)}(1-t)}$
is $-\delta f^{(2m-1)}$,  attained at $t=0$, and its minimum on the same interval is $0$ and
it is attained at $t=1/2$. Consequently, using Proposition \ref{pr41}, we have
\begin{align*}
0\leq R_m&\leq(-1)^{m+1}\left(\int_0^{1/2}\frac{ B_{2m-1}(t)}{(2m-1)!}\,dt\right)\,\delta f^{(2m-1)}\\
&\leq(-1)^{m+1}\frac{B_{2m}(1/2)-B_{2m}(0)}{(2m)!}\,\delta f^{(2m-1)}\\
&\leq(2-2^{1-2m})\frac{\abs{b_{2m}}}{(2m)!}\abs{\delta f^{(2m-1)}}\\
&\leq4\left(1-\frac{1}{2^{2m}}\right)\left(1+\frac{3}{2^{2m}}\right)\frac{1}{(2\pi)^{2n}}\abs{\delta f^{(2m-1)}}\\
&\leq4\frac{1+2^{1-2m}}{(2\pi)^{2m}}\abs{\delta f^{(2m-1)}}\leq \frac{6}{(2\pi)^{2m}}\abs{\delta f^{(2m-1)}}\\
\end{align*}
and the desired conclusion follows.
\end{proof}
\bg

\qquad Before proceeding to the next result, we will prove the following property that generalises the well-known
``Riemann Lebesgue's lemma''.
\bg
\begin{lemma}\label{lm62}
Consider  an integrable function $h:[0,1]\vers \comp$, and  a piecewise continuous $1$-periodic function
$g:\reel\vers\comp$. Then
\[
\lim_{p\to\infty}\int_0^1g(pt)h(t)\,dt=\left(\int_0^1g(t)\,dt\right)\cdot\left(\int_0^1h(t)\,dt\right)
\]
\end{lemma}
\begin{proof}
First, suppose that $\int_0^1g=0$. This implies that $x\mapsto G(x)=\int_0^xg(t)dt$ is a continuous $1$-periodic function.
Particularly $G$ is bounded, and we can define $M=\sup_{\reel}\abs{G}$.\bg
\qquad Now, assume that $h=\chi_{[\alpha,\beta)}$; the characteristic function of an interval $[\alpha,\beta)$.
In this case
\[
\int_0^1g(pt)h(t)\,dt=\int_\alpha^\beta g(pt)\,dt=\frac{G(p\beta)-G(p\alpha)}{p}
\]
So
\[
\abs{\int_0^1g(pt)h(t)\,dt}\leq\frac{2M}{p}
\]
and consequently $\lim\limits_{p\to\infty}\int_0^1g(pt)h(t)\,dt=0$. 
\bg\qquad
Using linearity, we see that the same conclusion holds if
$h$ is a step  function. Finally, the density of the space of step functions in $L^1([0,1])$ implies that
$\lim\limits_{p\to\infty}\int_0^1g(pt)h(t)\,dt=0$ for every integrable function $h:[0,1]\vers \comp$.\bg
\qquad Applying the preceding case to the function $\tilde{g}=g-\int_0^1g$, that satisfies $\int_0^1\tilde{g}(t)dt=0$, we
conclude that
\[
\lim_{p\to\infty}\int_0^1\tilde{g}(pt)h(t)\,dt=0
\]
for every $h\in L^1([0,1])$. Which is the desired conclusion.
\end{proof}
\bg
\qquad The next theorem is the main result of this section.
\bg
\begin{theorem}\label{th63} For a positive integer $p$, and a function $f$ having at least a
  continuous $m^{\text{th}}$ derivative on the interval $[0,1]$, we
define the quantities $\mathcal{H}_p(f;x)$ and $\mathcal{E}(p,m,f;x)$ for $x\in[0,1]$ by
\begin{align*}
\mathcal{H}_p(f;x)&=\frac{1}{p}\sum_{k=0}^{p-1}f\left(\frac{k+x}{p}\right)\\
\noalign{\text{and}}\\
\mathcal{E}(p,m,f;x)&=\int_0^1f(t)\,dt-\mathcal{H}_p(f;x)+\sum_{k=1}^m\frac{B_k(x)}{k!}\cdot\frac{\delta f^{(k-1)}}{p^k}.
\end{align*}
Then,
\begin{enumeratei}
\item\label{th631} The quantity $\mathcal{E}(p,m,f;x)$ has the following expression in terms of $\widetilde{B}_m$:
\[
\mathcal{E}(p,m,f;x)=\frac{1}{p^m}\int_0^1\frac{\widetilde{B}_m(x-pt)}{m!}f^{(m)}(t)\,dt.
\]
\item\label{th632} It satisfies also the following inequality:
\[
\abs{\mathcal{E}(p,m,f;x)}\leq\frac{8}{\pi}\cdot\frac{1}{(2\pi p)^m}\cdot \sup_{[0,1]}\abs{f^{(m)}}.
\]
\item\label{th633} Moreover,
\[
\lim_{p\to\infty} p^m\cdot\mathcal{E}(p,m,f;x)=0.
\]
\end{enumeratei}
\end{theorem}
\begin{proof}
 Applying Proposition \ref{pr61} to the function $x\mapsto\mathcal{H}_p(f;x)$ we obtain
\begin{equation*}\label{E:star}
\int_0^1\mathcal{H}_p(f;t)\,dt-\mathcal{H}_p(f;x)+\sum_{k=0}^{m-1}\frac{B_{k+1}(x)}{(k+1)!}\,\delta \mathcal{H}^{(k)}_p(f;\cdot)=\frac{1}{m!}
\int_0^1\widetilde{B}_m(x-t) \mathcal{H}^{(m)}_p(f;t)\,dt.\tag{$*$}
\end{equation*}
But
\begin{equation*}
\int_0^1\mathcal{H}_p(f;t)\,dt =\frac{1}{p}\sum_{k=0}^{p-1}\int_0^1f\left(\frac{k+t}{p}\right)\,dt 
 =\sum_{k=0}^{p-1}\int_{k/p}^{(k+1)/p}f(u)du=\int_0^1f(u)\,du.
\end{equation*}
Also,
\[
\mathcal{H}^{(k)}_p(f;x)=\frac{1}{p^{k+1}}\sum_{k=0}^{p-1}f^{(k)}\left(\frac{k+x}{p}\right)=\frac{1}{p^k}\mathcal{H}_p(f^{(k)};x).
\]
Thus
\begin{align*}
\delta \mathcal{H}^{(k)}_p(f;\cdot)&=\frac{1}{p^{k+1}}\left(\sum_{k=0}^{p-1}f^{(k)}\left(\frac{k+1}{p}\right)
-\sum_{k=0}^{p-1}f^{(k)}\left(\frac{k}{p}\right)\right),\\
&=\frac{f^{(k)}(1)-f^{(k)}(0)}{p^{k+1}}=\frac{1}{p^{k+1}}\delta f^{(k)}.
\end{align*}
Replacing the above results in \eqref{E:star} we conclude that
\begin{equation*}
\mathcal{E}(p,m,f;x)=\frac{1}{m!\cdot p^m}\int_0^1\widetilde{B}_m(x-t)\mathcal{H}_p(f^{(m)};t)\,dt,
\end{equation*}
and \itemref{th631} follows, since
\begin{align*}
\int_0^1\widetilde{B}_m(x-t)\mathcal{H}_p(f^{(m)};t)\,dt&=
\frac{1}{p}\sum_{k=0}^{p-1}\int_0^1\widetilde{B}_m(x-t)f^{(m)}\left(\frac{k+t}{p}\right)\,dt\\
&=\sum_{k=0}^{p-1}\int_{k/p}^{(k+1)/p}\widetilde{B}_m(x+k-pt)f^{(m)}(t)\,dt\\
&=\sum_{k=0}^{p-1}\int_{k/p}^{(k+1)/p}\widetilde{B}_m(x-pt)f^{(m)}(t)\,dt\\
&=\int_{0}^{1}\widetilde{B}_m(x-pt)f^{(m)}(t)\,dt
\end{align*}
\qquad Using \itemref{th631}, and recalling that $\widetilde{B}_m$ is $1$-periodic, we see that
\begin{align*}
\abs{\mathcal{E}(p,m,f;x)}&\leq\frac{1}{m!\cdot p^m}\int_{0}^{1}\abs{\widetilde{B}_m(x-pt)}\cdot\abs{f^{(m)}(t)}\,dt\\
&\leq\frac{1}{m!\cdot p^m}\sup_{t\in[0,1]}\abs{f^{(m)}(t)}\cdot\int_{0}^{1}\abs{\widetilde{B}_m(x-pt)}\,dt\\
&\leq\frac{1}{m!\cdot p^{m}}\sup_{t\in[0,1]}\abs{f^{(m)}(t)}\cdot\frac{1}{p}\int_{0}^{p}\abs{\widetilde{B}_m(x-u)}\,du\\
&=\frac{1}{m!\cdot p^{m}}\sup_{t\in[0,1]}\abs{f^{(m)}(t)}\cdot\int_{0}^{1}\abs{\widetilde{B}_m(t)}\,dt,
\end{align*}
and \itemref{th632} follows using Proposition~\ref{pr44} \itemref{pr441}.\bg
\qquad Finally, applying Lemma \ref{lm62} to the $1$-periodic function $u\mapsto g(t)=\widetilde{B}_m(x-t)$ and
the integrable function $t\mapsto h(t)=f^{(m)}(t)$ we obtain \itemref{th633} because $\int_0^1g=0$  in this case.
\end{proof}
\bg
\section{Asymptotic expansions for numerical quadrature formul\ae}\label{sec7}
\bn
\qquad In this section we only consider functions defined on the intervall $[0,1]$. The more general
case of a functions defined on $[a,b]$ can be obtained by applying the  results after using
the change of variable $t\mapsto  a+t(b-a)$. \bg
\qquad We consider a function $f:[0,1]\vers\comp$, having a continuous $m^{\text{th}}$ derivative, with $m\geq2$, 
and we will use freely the notation of the previous section.
\bg
\subsection{Riemann sums.} \label{sbsec71} The Riemann sum of $f$ obtained by taking the values of the function $f$ at the lower bound of each subdivision interval, is given by
\begin{equation}\label{E:RL1}
\mathcal{R}_p^L(f)=\frac{1}{p}\sum_{k=0}^{p-1}f\left(\frac{k}{p}\right).
\end{equation}
 According to Theorem \ref{th63} we have $\mathcal{R}_p^L(f)=\mathcal{H}_p(f,0)$, so
\begin{equation}\label{E:RL2}
\int_0^1f(t)dt=\mathcal{R}_p^L(f)+\frac{\delta f}{2p}-\sum_{1\leq k\leq \frac{m}{2}}\frac{b_{2k}}{(2k)!\cdot p^{2k}}\cdot\delta f^{(2k-1)}+
\mathcal{E}(p,m,f;0)
\end{equation}
\qquad Similarly, the Riemann sum of $f$ obtained by taking the values of the function $f$ at the upper bound
of each subdivision interval, is given by
\begin{equation}\label{E:RR1}
\mathcal{R}_p^R(f)=\frac{1}{p}\sum_{k=1}^{p}f\left(\frac{k}{p}\right)
\end{equation}
 And again using Theorem \ref{th63} we have $\mathcal{R}_p^R(f)=\mathcal{H}_p(f,1)$, so
\begin{equation}\label{E:RR2}
\int_0^1f(t)dt=\mathcal{R}_p^R(f)-\frac{\delta f}{2p}-\sum_{1\leq k\leq \frac{m}{2}}\frac{b_{2k}}{(2k)!\cdot p^{2k}}\cdot\delta f^{(2k-1)}+
\mathcal{E}(p,m,f;0)
\end{equation}
where we noted that $\mathcal{E}(p,m,f;1)=\mathcal{E}(p,m,f;0)$, since $\widetilde{B}_m$ is $1$-periodic.
\bg
\qquad Also, the Riemann sum of $f$ obtained by taking the values of the function $f$ at the midpoint
of each subdivision interval, is given by
\begin{equation}\label{E:RM1}
\mathcal{R}_p^M(f)=\frac{1}{p}\sum_{k=0}^{p-1}f\left(\frac{2k+1}{2p}\right)
\end{equation}
 This is the ``Midpoint Quadrature Rule''. By Theorem \ref{th63} we have $\mathcal{R}_p^M(f)=\mathcal{H}_p\left(f,\frac12\right)$, so
\begin{equation}\label{E:RM2}
\int_0^1f(t)dt=\mathcal{R}_p^M(f)-\sum_{1\leq k\leq \frac{m}{2}}\frac{(2^{1-2k}-1)b_{2k}}{(2k)!\cdot p^{2k}}\cdot\delta f^{(2k-1)}+
\mathcal{E}\left(p,m,f;\frac12\right)
\end{equation}
\qquad For example, taking $m=2$, we obtain from Theorem \ref{th63} \eqref{th633}:
\begin{equation}\label{E:RM3}
\lim_{p\to\infty} p^2\left(\int_0^1f(t)dt-\mathcal{R}_p^M(f)\right)=\frac{f'(1)-f'(0)}{24}.
\end{equation}
Thus, the midpoint quadrature rule is a second order rule.
\bg
\subsection{The trapezoidal rule.} \label{sbsec72} Taking the half sum of $\mathcal{R}_p^L(f)$ and $\mathcal{R}_p^R(f)$ we obtain the
trapesoidal rule that corresponds to approximating $f$ linearly on each interval of the subdivision.
\begin{equation}\label{E:T1}
\mathcal{T}_p(f)=\frac{1}{2}\left(\mathcal{R}_p^L(f)+\mathcal{R}_p^R(f)\right)=\frac{f(0)+f(1)}{2p}+
\frac{1}{p}\sum_{k=1}^{p-1}f\left(\frac{k}{p}\right)
\end{equation}
Using \eqref{E:RL2} and \eqref{E:RR2} we see that
\begin{equation}\label{E:T2}
\int_0^1f(t)dt=\mathcal{T}_p(f)-\sum_{1\leq k\leq \frac{m}{2}}\frac{b_{2k}}{(2k)!\cdot p^{2k}}\cdot\delta f^{(2k-1)}+\mathcal{E}(p,m,f;0)
\end{equation}
In particular, choosing $m=2$, we obtain from Theorem \ref{th63} \eqref{th633}:
\begin{equation}\label{E:T3}
\lim_{p\to\infty} p^2\left(\int_0^1f(t)dt-\mathcal{T}_p(f)\right)=-\frac{f'(1)-f'(0)}{12}.
\end{equation}
Thus, the trapezoidal quadrature rule is a second order rule.
\bg
\qquad In fact, for the case of the trapezoidal rule we have a more refined result in some cases. This is the object of the following proposition.
\bg
\begin{proposition}\label{pr71}
Consider a positive integer $m$, and a function $f$ that has a  continuous $(2m-1)^{\text{st}}$ derivative on $[0,1]$. If
 $f^{(2m-1)}$ is {\normalfont \text{decreasing}}, then, for every positive integer $p$ we have
\[
\int_0^1f(t)\,dt=\mathcal{T}_p(f)-
\sum_{k=1}^{m-1}\frac{b_{2k}}{(2k)!\cdot p^{2k}}\cdot\delta f^{(2k-1)}+(-1)^{m+1}R_{m,p}
\]
with
\[
0\leq R_{m,p}\leq\frac{6}{(2\pi p)^{2m}}\left(f^{(2m-1)}(0)-f^{(2m-1)}(1)\right).
\]
\end{proposition}
\begin{proof}
Our starting point will be Corollary \ref{cor61} applied to the function 
$x\mapsto f\left(\frac{j+x}{p}\right)$, with $0\leq j<p$. It follows that
\begin{align*}
p\int_{j/p}^{(j+1)/p}f(t)\,dt&=\frac{1}{2}\left(f\left(\tfrac{j+1}{p}\right)+f\left(\tfrac{j}{p}\right)\right)\\
&{}-\sum_{k=1}^{m-1}\frac{b_{2k}}{(2k)!p^{2k-1}}\,\left(f^{(2k-1)}\left(\tfrac{j+1}{p}\right)-f^{(2k-1)}\left(\tfrac{j}{p}\right)\right)+(-1)^{m+1}R_{m,p,j} 
\end{align*}
with
\begin{equation*}\label{E:T4}
0\leq R_{m,p,j}\leq\frac{6}{(2\pi)^{2m}p^{2m-1}}\left(f^{(2m-1)}\left(\tfrac{j}{p}\right)-f^{(2m-1)}\left(\tfrac{j+1}{p}\right)\right).\tag{*}
\end{equation*}
Adding these inequalities, for $0\leq j<p$, and recalling \eqref{E:T1} we see that
\[
\int_0^1f(t)\,dt=\mathcal{T}_p(f)
-\sum_{k=1}^{m-1}\frac{b_{2k}}{(2k)!p^{2k}}\,\delta f^{(2k-1)}+(-1)^{m+1}R_{m,p}
\]
with $R_{m,p}=\frac{1}{p}\sum_{j=0}^{p-1}R_{m,p,j}$. Now, using \eqref{E:T4} we get
\[
0\leq R_{m,p} \leq\frac{6}{(2\pi p)^{2m}}\left(f^{2m-1}(0)-f^{(2m-1}(1)\right)
\]
which is the desired conclusion.
\end{proof}
\bg
\begin{application} \label{ap7}\textbf{An asymptotic expansion for a trigonometric sum.}\medskip\nobreak
\qquad For a positive integer $p$,  we consider the trigonometric sum
\[J_p=\sum_{j=1}^{p-1}j\cot\left(\frac{j\pi}{p}\right).\]
This sum will be studied in detail later, but we want here to illustrate the use of the Proposition~\ref{pr71}. \bg
\begin{proposition}\label{pr72} For every positive integers $p$ and $m$, there is a real number $\theta_{p,m}$ such that
\[
J_p=-\frac{1}{\pi} p^2H_p+\frac{\ln(2\pi)}{\pi}p^2-\frac{p}{2\pi}-
\sum_{k=1}^{m-1}\frac{b_{2k}(1+2\zeta(2k))}{2\pi k\cdot p^{2k-2}}+\frac{(-1)^{m}}{p^{2m-2}}\theta_{p,m}
\]
and
\[0<\theta_{p,m}<\dfrac{\abs{b_{2m}}(1+2\zeta(2m))}{2\pi  m},\]
where
$H_p=\sum_{j=1}^{p}1/j$ is the $p^{\text{th}}$ harmonic number.
\end{proposition}
\begin{proof}
Indeed, let $\vf$ be the function defined by 
\[\vf(x)=\pi x\cot(\pi x)+\frac{1}{1-x}.\]
According to formula \eqref{E:cotf} we know that 
\[
\vf(x)=2+\frac{x}{x+1}+\sum_{n=2}^\infty\left(\frac{x}{x-n}+\frac{x}{x+n}\right).
\]
Thus, $\vf$ is defined and analytic on the interval $(-1,2)$. Let us show that, for every positive integer $k$, the derivative
$\vf^{(2k)}$ is negative on the interval $[0,1]$. To this end, we note, using \eqref{E:cots}, that
\begin{align*}
\vf(x)&=\pi x\cot(\pi x)+\frac{2}{1-x^2}-\frac{1}{1+x}\\
&=3-\frac{1}{1+x}-2\sum_{n=1}^\infty (\zeta(2n)-1)x^{2n}.
\end{align*}
Hence,
\[
\frac{\vf^{(2k)}(x)}{(2k)!}=-\frac{1}{(1+x)^{2k+1}}-2\sum_{n=k}^\infty \com{2n}{2k}(\zeta(2n)-1)x^{2n-2k},
\]
which is clearly negative on $[0,1]$.\bg
Now, we can apply Proposition~\ref{pr71} to $\vf$. We only need to calculate $\delta \vf^{(2k-1)}$ for every $k$. Note that
\[
\vf(x)+\vf(1-x)=\frac{1-\pi\cot(\pi x)}{x} +2\pi x\cot(\pi x)+\frac{1}{1-x}
\]
Thus, using \eqref{E:cots} again we obtain, for $\abs{x}<1$,
\[
\vf(x)+\vf(1-x)=3+ \sum_{n=1}^\infty(2\zeta(2n)+1)x^{2n-1}+\sum_{n=1}^\infty(1-4\zeta(2n))x^{2n}
\]
Taking the $(2k-1)^{\text{st}}$ derivative at $x=0$, we get
\[
\frac{\delta \vf^{2k-1}}{(2k-1)!}=-1-2\zeta(2k).
\]
So, applying Proposition~\ref{pr71}, we obtain
\begin{equation*}\label{E:pr921}
\int_0^1\vf(t)\,dt=\mathcal{T}_p(\vf)+\sum_{k=1}^{m-1}\frac{b_{2k}(1+2\zeta(2k))}{2k\cdot p^{2k}}+(-1)^{m+1}R_{m,p}\tag{*}
\end{equation*}
with
\[
0\leq R_{m,p}\leq\frac{6(2m-1)!\cdot (1+2\zeta(2m))}{(2\pi p)^{2m}}
\]
But
\begin{align*}
\mathcal{T}_p(\vf)&=\frac{\vf(0)+\vf(1)}{2p}+\frac{1}{p}\sum_{j=1}^{p-1}\vf\left(\frac{j}{p}\right)\\
&=\frac{3}{2p}+\frac{\pi}{p^2}\sum_{j=1}^{p-1} j\cot\left(\frac{\pi j}{p}\right)+ \sum_{j=1}^{p-1}\frac{1}{p-j}=H_p+\frac{1}{2p}+\frac{\pi}{p^2}J_p
\end{align*}
Also, for $x\in[0,1)$, we have
\[
\int_0^x\vf(t)\,dt=-\ln(1-x)+x\ln \sin(\pi x)-\int_0^x \ln\sin(\pi t)\,dt
\]
and, letting $x$ tend to $1$ we obtain
\[
\int_0^1\vf(t)\,dt=\ln(\pi)-\int_0^1\ln\sin(\pi t)\,dt=\ln(2\pi)
\]
where we used the fact   $\int_0^1\ln\sin(\pi t)\,dt=-\ln2$, (see \cite[4.224 Formula 3.]{grad}. Thus
\eqref{E:pr921} is equivalent to
\[
\ln(2\pi) =H_{p}+\frac{1}{2p}+\frac{\pi}{p^2} J_p+\sum_{k=1}^{m-1}\frac{b_{2k}(1+2\zeta(2k))}{2k\cdot p^{2k}}+(-1)^{m+1}R_{m,p}
\]
or
\[\frac{\pi}{p^2} J_p=\ln(2\pi) -H_{p}-\frac{1}{2p}-\sum_{k=1}^{m-1}\frac{b_{2k}(1+2\zeta(2k))}{p^{2k}}+(-1)^{m}R_{m,p}
\]
Thus, we have shown that for every nonnegative integer $m$ we have
\begin{align*}
\frac{\pi}{p^2} J_p&<\ln(2\pi) -H_{p}-\frac{1}{2p}-\sum_{k=1}^{2m}\frac{b_{2k}(1+2\zeta(2k))}{2k\cdot p^{2k}}\\
\noalign{\text{and}}
\frac{\pi}{p^2} J_p&>\ln(2\pi) -H_{p}-\frac{1}{2p}-\sum_{k=1}^{2m+1}\frac{b_{2k}(1+2\zeta(2k))}{2k\cdot p^{2k}}
\end{align*}
So, for every positive integer $m$ we have, 
\[
0<(-1)^{m}\left( J_p-\frac{p^2}{\pi}(\ln(2\pi) -H_{p})+\frac{p}{2\pi}+\sum_{k=1}^{m-1}\frac{b_{2k}(1+2\zeta(2k))}{2\pi k\cdot p^{2k-2}}\right)<\frac{\abs{b_{2m}}(1+2\zeta(2m))}{2\pi m\cdot p^{2m-2}}.
\]
Which is the desired conclusion.
\end{proof}
\bg
\qquad The result of this proposition is not completely satisfactory, because of the sum $H_p$.  That is why it is just the beginning of the story! It will be 
pursued in a later section.\bg
\end{application}

\subsection{Simpson's  rule.} \label{sbsec73} Comparing \eqref{E:T3} and \eqref{E:RM3} we see that 
\[\lim_{p\to\infty}p^2\left(\int_0^1f-\frac{\mathcal{T}_p(f)+2\mathcal{R}_p^M(f)}{3} \right)=0,\]
so the quantity
$\frac13\left(\mathcal{T}_p(f)+2\mathcal{R}_p^M(f)\right)$ is a better quarature rule than the second order ones. Hence, let us define
the ``Simpson quadrature rule'' by
\begin{align}\label{E:S1}
\mathcal{S}_p(f)&=\frac{1}{3}\left(\mathcal{T}_p^L(f)+2\mathcal{R}_p^M(f)\right)\notag\\
&=\frac{1}{6p}\sum_{k=1}^{p-1}\left(f\left(\frac{2k}{2p}\right)+f\left(\frac{2k+1}{2p}\right)+f\left(\frac{2k+2}{2p}\right)\right)
\end{align}
Using \eqref{E:T2} and \eqref{E:RM2} we see that
\begin{equation}\label{E:S2}
\int_0^1f(t)dt=\mathcal{S}_p(f)+\sum_{2\leq k\leq \frac{m}{2}}\frac{(1-4^{1-k})b_{2k}}{3(2k)!\cdot p^{2k}}\cdot\delta f^{(2k-1)}+
\mathcal{E}^S(p,m,f)
\end{equation}
with
\[
\mathcal{E}^S(p,m,f)=\frac{1}{3}\left(\mathcal{E}(p,m,f;0)+2\mathcal{E}\left(p,m,f;\frac12\right)\right)
\]
Using Theorem \ref{th63} we see that
\[
\mathcal{E}^S(p,m,f) \leq\frac{8}{\pi}\cdot\frac{1}{(2p)^{2m}}\cdot\sup_{[0,1]}\abs{f^{(2m)}}
\quad\text{and}\quad \lim_{p\to\infty}p^{2m}\cdot\mathcal{E}^S(p,m,f)=0
\]
In particular, choosing $m=4$, we obtain :
\begin{equation}\label{E:S3}
\lim_{p\to\infty} p^4\left(\int_0^1f(t)dt-\mathcal{S}_p(f)\right)=-\frac{f^{(3)}(1)-f^{(3)}(0)}{2880}.
\end{equation}
Thus, the Simpson quadrature rule is a forth order rule.
\bg
\subsection{The two point Gauss rule.} \label{sbsec74} Applying Theorem~\ref{th63} at $x$ and $1-x$ and using
Proposition~\ref{pr21} \itemref{pr212}, we obtain after taking the half sum:
\[
\int_0^1f(t)dt=\frac{1}{2p}\sum_{k=0}^{p-1}\left(f\left(\tfrac{k+x}{p}\right)+f\left(\tfrac{k+1-x}{p}\right)\right)
-\sum_{1\leq k\leq\frac{m}{2}}\frac{B_{2k}(x)}{(2k)!\,p^{2k}}\delta f^{(2k-1)}+\widetilde{\mathcal{E}}(p,m,f;x)
\]
with
\[
\widetilde{\mathcal{E}}(p,m,f;x)=\frac12\left(\mathcal{E}(p,m,f;x)+\mathcal{E}(p,m,f;1-x)\right).
\]
Here $\widetilde{\mathcal{E}}$ satisfies the same properties as $\mathcal{E}$ in Theorem~\ref{th63}.
The case $x=0$ corresponds to the trapezoidal rule, and the case $x=\frac12$ corresponds the midpoint point rule.
But the best choice for $x$ is when $x=\alpha=\frac12-\frac{1}{\sqrt{12}}$ which is a zero of $B_2$. Then, we obtain
``the two point Gauss quadrature rule'':
\begin{equation}\label{E:G1}
\mathcal{G}_p(f)=\frac{1}{2p}\sum_{k=0}^{p-1}\left(f\left(\frac{k+\frac{1}{2}-\frac{1}{\sqrt{12}}}{p}\right)+
f\left(\frac{k+\frac{1}{2}+\frac{1}{\sqrt{12}}}{p}\right)\right)
\end{equation}
with
\begin{equation}\label{E:G2}
\int_0^1f(t)dt=\mathcal{G}_p(f)-\sum_{2\leq k\leq\frac{m}{2}}\frac{B_{2k}(\alpha)}{(2k)!\,p^{2k}}\delta f^{(2k-1)}+\widetilde{\mathcal{E}}(p,m,f;\alpha)
\end{equation}
For example, with $m=4$ we find that
\begin{equation}\label{E:G3}
\lim_{p\to\infty}p^4\left(\int_0^1f(t)dt-\mathcal{G}_p(f)\right)=\frac{f^{(3)}(1)-f^{(3)}(0)}{4320}
\end{equation}
Thus, the two point Gauss quadrature rule is a forth order rule.
\bg 
\newcommand{\ro}[2]{\,\mathcal{T}^{(#1)}_{#2}(f)}
\newcommand{\roe}[2]{\,\mathcal{E}^{(#1)}_m({#2})}
\subsection{Romberg's rule.} \label{sbsec75} Let us consider again the case of the trapezoidal rule \eqref{E:T1}, and the error asymptotic
expansion:
\begin{equation}\label{E:Ro1}
\int_0^1f(t)dt=\mathcal{T}_p(f)-\sum_{1\leq k\leq \frac{m}{2}}\frac{b_{2k}}{(2k)!\cdot p^{2k}}\cdot\delta f^{(2k-1)}+\mathcal{E}(p,m,f;0)
\end{equation}
We define $\ro{0}{p}=\mathcal{T}_p(f)$ and for simplicity we write $\roe{0}{p}$
for $\mathcal{E}(p,m,f;0)$. Next, we define inductively
\begin{align*}
\ro{\ell}{p}&=\frac{4^\ell \ro{\ell-1}{2p}-\ro{\ell-1}{p}}{4^\ell-1}\\
\roe{\ell}{p}&=\frac{4^\ell \roe{\ell-1}{2p}-\roe{\ell-1}{p}}{4^\ell-1}
\end{align*}
for $\ell=1,2,\ldots$. 
\bg
It is easy to prove by induction, starting from \eqref{E:Ro1} that
\begin{align*}
\int_0^1f(t)dt&=\ro{1}{p}-\sum_{1\leq k\leq \frac{m}{2}}\left(\frac{4^{1-k}-1}{4-1}\right)
\frac{b_{2k}}{(2k)!\cdot p^{2k}}\,\delta f^{(2k-1)}+\roe{1}{p}\\
\int_0^1f(t)dt&=\ro{2}{p}-\sum_{1\leq k\leq \frac{m}{2}}\left(\frac{4^{1-k}-1}{4-1}\right)\left(\frac{4^{2-k}-1}{4^2-1}\right)
\frac{b_{2k}}{(2k)!\cdot p^{2k}}\,\delta f^{(2k-1)}+\roe{2}{p}\\
\vdots\qquad&=\qquad\vdots\qquad\qquad\qquad\qquad\qquad\vdots\\
\int_0^1f(t)dt&=\ro{\ell}{p}-\sum_{1\leq k\leq \frac{m}{2}}\prod_{j=1}^\ell\left(\frac{4^{j-k}-1}{4^j-1}\right)
\frac{b_{2k}}{(2k)!\cdot p^{2k}}\,\delta f^{(2k-1)}+\roe{\ell}{p}
\end{align*}
Note that
\[
\prod_{j=1}^\ell\left(\frac{4^{j-k}-1}{4^j-1}\right)=0\quad\text{for $k=1,2,\ldots,\ell$,}
\]
so, in fact, we have 
\begin{equation}\label{E:Ro3}
\int_0^1f(t)dt=\ro{\ell}{p}-\sum_{\ell< k\leq \frac{m}{2}}\prod_{j=1}^\ell\left(\frac{4^{j-k}-1}{4^j-1}\right)
\frac{b_{2k}}{(2k)!\cdot p^{2k}}\,\delta f^{(2k-1)}+\roe{\ell}{p}
\end{equation}
In order to simplfy a little bit the notation we recall that the finite $q$-Pochhammer $(z;q)_n$ symbol is defined as the product
\[
(z;q)_n=\prod_{k=1}^n\left(1-z q^{k-1}\right).
\]
The limit as $n$ tend to $+\infty$ defines the $q$-Pochhammer symbol $(z;q)_\infty$ when $\abs{q}<1$. Also, we define the 
$q$-binomial coefficient $\com{n}{m}_q$ by the formula
\begin{equation}\label{E:qbin}
\com{n}{m}_q=\frac{(q;q)_n}{(q;q)_{n-m}(q;q)_m},\quad\text{for $0\leq m\leq n$.}
\end{equation}
With this notation we see that for $k>\ell$ and $q=1/4$, we have
\begin{align*}
\prod_{j=1}^\ell\left(\frac{4^{j-k}-1}{4^j-1}\right)
&=\frac{(-1)^\ell}{2^{\ell(1+\ell)}}\prod_{j=1}^\ell\left(\frac{1-q^{k-j}}{1-q^j}\right)\\
&=\frac{(-1)^\ell}{2^{\ell(1+\ell)}}\frac{(q;q)_{k-1}}{(q;q)_{k-1-\ell}(q;q)_\ell}=\frac{(-1)^\ell}{2^{\ell(1+\ell)}}\com{k-1}{\ell}_q
\end{align*}
Thus, we can write \eqref{E:Ro3} as follows
\begin{equation}\label{E:Ro4}
\int_0^1f(t)dt=\ro{\ell}{p}-\frac{(-1)^\ell}{2^{\ell(1+\ell)}}\sum_{\ell< k\leq \frac{m}{2}}\com{k-1}{\ell}_{\frac14}
\,\frac{b_{2k}}{(2k)!\cdot p^{2k}}\,\delta f^{(2k-1)}+\roe{\ell}{p}
\end{equation}
Also, since according to Theorem~\ref{th63}~\itemref{th633}, we have $\lim\limits_{p\to\infty}p^m\roe{0}{p}=0$, we conclude by
induction on $\ell$ that we  have
\[
\lim_{p\to\infty}p^m\roe{\ell}{p}=0,\quad\text{for $\ell=0,1,2,\ldots$.}
\]
In particular, when $m\geq 2\ell+2$ we have
\[
\lim_{p\to\infty}p^{2\ell+2}\left(\int_0^1f(t)dt-
\ro{\ell}{p}\right)=-\frac{\abs{b_{2\ell+2}}}{2^{\ell(1+\ell)}\,(2\ell+2)!}\,\delta f^{(2\ell+1)}.
\]
On the other hand, using Theorem~\ref{th63}~\itemref{th632} we have
\[
\abs{\roe{0}{p}}\leq\frac{8}{(2\pi p)^m}\cdot\sup_{[0,1]}\abs{f^{(m)}}
\]
and, for $\ell=1,2,\ldots$ we have
\[
\abs{\roe{\ell}{p}}\leq\frac{4^\ell \abs{\roe{\ell-1}{2p}}+\abs{\roe{\ell-1}{p}}}{4^\ell-1}
\]
So
\begin{align}
\abs{\roe{1}{p}}&\leq\frac{8}{(2\pi p)^m}\left(\frac{1+4/2^m}{4-1}\right)\cdot\sup_{[0,1]}\abs{f^{(m)}}\notag\\
\abs{\roe{2}{p}}&\leq\frac{8}{(2\pi p)^m}\left(\frac{1+4/2^m}{4-1}\right)\left(\frac{1+4^2/2^m}{4^2-1}\right)
\cdot\sup_{[0,1]}\abs{f^{(m)}}\notag\\
\vdots\qquad&\qquad\qquad\qquad\qquad\qquad \vdots\notag\\
\abs{\roe{\ell}{p}}&\leq\frac{8}{(2\pi p)^m}\prod_{j=1}^{\ell}\left(\frac{1+4^j/2^m}{ 4^j-1}\right)
\cdot\sup_{[0,1]}\abs{f^{(m)}}\label{E:Ro5}
\end{align}
But, for $m\geq 2\ell+2$ we have
\[
\prod_{j=1}^{\ell}\left(1+4^j/2^m\right)\leq\prod_{j=1}^{\ell}\left(1+4^{-(\ell-j+1)}\right)
=\prod_{j=1}^{\ell}\left(1+ 4^{-j}\right)
\]
So
\begin{equation}\label{E:Ro6}
\prod_{j=1}^{\ell}\left(\frac{1+4^j/2^m}{ 4^j-1}\right)\leq\frac{1}{2^{\ell(1+\ell)}}\prod_{j=1}^\ell\frac{1+4^{-j}}{1-4^{-j}}
\end{equation}
Now, since $x\mapsto\ln\left(\frac{1+x}{1-x}\right)$ is convex on the interval $[0,1/4]$ we conclude that
for $x\in[0,1/4]$ we have
\[
\ln\left(\frac{1+x}{1-x}\right)\leq 4\ln\left(\frac{1+4^{-1}}{1-4^{-1}}\right)x=4\ln\left(\frac{5}{3}\right)x
\]
Thus
\[
\sum_{j=1}^\ell\ln\left(\frac{1+4^{-j}}{1-4^{-j}}\right)<4\ln\left(\frac{5}{3}\right)\sum_{j=1}^\infty\frac{1}{4^j}=\frac{4}{3}\ln\left(\frac{5}{3}\right)
\]
Finally, since $(5/3)^{4/3}<2$ we obtain from \eqref{E:Ro6} that
\[
\forall\,m\geq 2\ell+2,\qquad \prod_{j=1}^{\ell}\left(\frac{1+4^j/2^m}{ 4^j-1}\right)\leq\frac{2}{2^{\ell(1+\ell)}}
\]
Thus, \eqref{E:Ro5} implies the following more appealing form
\begin{equation}\label{E:Ro7}
\abs{\roe{\ell}{p}}\leq\frac{16}{\pi}\cdot\frac{1}{2^{\ell(1+\ell)}(2\pi p)^m}\,\sup_{[0,1]}\abs{f^{(m)}}
\end{equation}
\qquad We have proved the following result:
\begin{proposition}\label{prRo} Let $m$ be a positive integer, and let $f$ be a function having a continuous $m^{\text{th}}$ derivative
on $[0,1]$. For a positive $p$, let $\mathcal{T}_p^{(0)}(f)$ be the trapezoidal quadrature rule
applied to $f$ defined by \eqref{E:T1}. Next, for $\ell\geq1$, define inductively the Romberg's rule of order $\ell$  by
\[
\ro{\ell}{p}=\frac{4^\ell \ro{\ell-1}{2p}-\ro{\ell-1}{p}}{4^\ell-1}
\]
Then
\[
\int_0^1f(t)dt=\ro{\ell}{p}-\frac{(-1)^\ell}{2^{\ell(1+\ell)}}\sum_{\ell< k\leq \frac{m}{2}}\com{k-1}{\ell}_{\frac14}
\,\frac{b_{2k}}{(2k)!\cdot p^{2k}}\,\delta f^{(2k-1)}+\roe{\ell}{p}
\]
with $\lim\limits_{p\to\infty}p^m\roe{\ell}{p}=0$, (where the $q$-binomial is defined by \eqref{E:qbin}.)
Moreover, for $m\geq2\ell+2$ we have
\[
\abs{\roe{\ell}{p}}\leq\frac{16}{\pi}\cdot\frac{1}{2^{\ell(1+\ell)}(2\pi p)^m}\,\sup_{[0,1]}\abs{f^{(m)}}.
\]
\end{proposition}
\qquad Consider the particular case where $m=2\ell+2$. In this case, Proposition \ref{prRo} implies
\begin{align*}
\abs{\int_0^1f(t)dt-\ro{\ell}{p}}&\leq\frac{1}{2^{\ell(1+\ell)}(2\pi p)^{2\ell+2}}\,\left(\frac{(2\pi)^{2\ell+2}\abs{b_{2\ell+2}}}{(2\ell+2)!}\,\abs{\delta f^{(2\ell+1)}}+\frac{16}{\pi}\,\sup_{[0,1]}\abs{f^{(2\ell+2)}}\right)\\
 &\leq\frac{1}{2^{\ell(\ell+1)}(2\pi p)^{2\ell+2}}\,\left(4\,\abs{\delta f^{(2\ell+1)}}+
\frac{16}{\pi}\,\sup_{[0,1]}\abs{f^{(2\ell+2)}}\right)\\
&\leq\frac{4+16/\pi}{2^{\ell(\ell+1)}(2\pi p)^{2\ell+2}}\,\sup_{[0,1]}\abs{f^{(2\ell+2)}}\\
&\leq\frac{10}{2^{(\ell+2)(\ell+1)}(\pi p)^{2\ell+2}}\,\sup_{[0,1]}\abs{f^{(2\ell+2)}}
\end{align*}
where we used Proposition~\ref{pr41} \itemref{pr411}, and the fact that $\abs{\delta f^{(2\ell+1)}}\leq\sup_{[0,1]}\abs{f^{(2\ell+2)}}$.\bg
Note that the Romberg's rule $\mathcal{T}_p^{(\ell)}$ subdivides the interval $[0,1]$ into $2^\ell p$ equal subintervals. In
the particular case $p=1$, we obtain
\[
\abs{\int_0^1f(t)dt-\ro{\ell}{1}}
\leq\frac{10}{2^{(\ell+2)(\ell+1)}\,\pi ^{2\ell+2}}\,\sup_{[0,1]}\abs{f^{(2\ell+2)}}.
\]

\bg
\section{Asymptotic expansions for the sum of certain series related to harmonic numbers}\label{sec8}
\bn
\qquad Recall that the sequence of  harmonic numbers $(H_n)_{n\in\nat}$ is defined by $H_n=\sum_{k=1}^n1/k$ (with the convention
$H_0=0$). It is well-known
that $\lim\limits_{n\to\infty}(H_n-\ln n)=\gamma$, where $\gamma\approx 0.57721\,56649$ is the so called Euler-Mascheroni Constant. \bg
 \qquad In the next proposition, the  asymptotic expansion of $(H_n)_{n\in\nat}$ is presented.
\bg
\begin{proposition}\label{pr81}
For every positive integer $n$ and nonnegative integer $m$, we have
\[
H_n=\ln n+\gamma+\frac{1}{2n}-\sum_{k=1}^{m-1}\frac{b_{2k}}{2k}\cdot\frac{1}{n^{2k}}+(-1)^m R_{n,m},
\]
with 
\[
R_{n,m}=\int_0^{1/2}\abs{B_{2m-1}(t)}\,
\sum_{j=n}^\infty\left(\frac{1}{(j+t)^{2m}}-\frac{1}{(j+1-t)^{2m}}
\right)\,dt
\]
Moreover,  $ 0<R_{n,m}<\dfrac{\abs{b_{2m}}}{2m\cdot n^{2m}}$.
\end{proposition}
\begin{proof}
Note that for $j\geq1$ we have
\[
\frac{1}{j}-\ln\left(1+\frac{1}{j}\right)=\int_0^1\left(\frac1j-\frac{1}{j+t}\right)\,dt=\int_0^1\frac{t}{j(j+t)}\,dt
\]
Adding these equalities as $j$ varies from $1$ to $n-1$ we conclude that
\[
H_n-\ln n-\frac{1}{n}=\int_0^1\left(\sum_{j=1}^{n-1}\frac{t}{j(j+t)}\right)\,dt.
\]
Thus, letting $n$ tend to $\infty$, and using  the monotone convergence theorem \cite[Corollary 2.3.5]{ath}, we conclude
\[
\gamma=\int_0^1\left(\sum_{j=1}^{\infty}\frac{t}{j(j+t)}\right)\,dt.
\]
It follows that
\begin{equation*}
\gamma+\ln n-H_n+\frac{1}{n}=\int_0^1\left(\sum_{j=n}^\infty\frac{t}{j(j+t)}\right)\,dt.
\end{equation*}
So, let us consider the function $f_n:[0,1]\vers\reel$ defined by
\[
f_n(t)=\sum_{j=n}^\infty\frac{t}{j(j+t)}
\]
Note that $f_n(0)=0$, $f_n(1)=1/n$, and that $f_n$ is infinitely continuously derivable with
\[
\frac{f_n^{(k)}(t)}{k!}=(-1)^{k+1}\sum_{j=n}^\infty\frac{1}{(j+t)^{k+1}},\quad\text{for $k\geq1$.}
\]
In particular, 
\[
\frac{f_n^{(2k-1)}(t)}{(2k-1)!}=\sum_{j=n}^\infty\frac{1}{(j+t)^{2k}},\quad\text{for $k\geq1$.}
\]
So, $f_n^{(2m-1)}$ is decreasing on the interval $[0,1]$, and
\[
\frac{\delta f_n^{(2k-1)}}{(2k-1)!} =\sum_{j=n}^\infty\frac{1}{(j+1)^{2k}}
-\sum_{j=n}^\infty\frac{1}{j^{2k}}=-\frac{1}{n^{2k}}
\]
Applying Corollary \ref{cor61} to $f_n$, and using the above data, we get
\[
\gamma+\ln n-H_n+\frac{1}{2n}=
\sum_{k=1}^{m-1}\frac{b_{2k}}{2k\,n^{2k}}+(-1)^{m+1}R_{n,m}
\]
with
\[
R_{n,m}=\int_0^{1/2}\abs{B_{2m-1}(t)}\,
\sum_{j=n}^\infty\left(\frac{1}{(j+t)^{2m}}-\frac{1}{(j+1-t)^{2m}}
\right)\,dt
\]
and
\[
0< R_{n,m}<\frac{6\cdot(2m-1)!}{(2\pi)^{2m}n^{2m}}.
\]
What is important in this estimate is the lower bound, \textit{i.e.} $R_{n,m}>0$. In fact, considering separately the cases $m$ odd and $m$ even, we obtain, for every nonnegative integer $m'$:
\begin{align*}
H_n&<\ln n+\gamma+\frac{1}{2n}-\sum_{k=1}^{2m'}\frac{b_{2k}}{2k}\cdot\frac{1}{n^{2k}},\\
\noalign{\text{and}}
H_n&>\ln n+\gamma+\frac{1}{2n}-\sum_{k=1}^{2m'+1}\frac{b_{2k}}{2k}\cdot\frac{1}{n^{2k}}.\\
\end{align*}
This yields the following more precise estimate for the error term:
\begin{equation}\label{kn}
0<(-1)^{m}\left(H_n-\ln n-\gamma-\frac{1}{2n}+
\sum_{k=1}^{m-1}\frac{b_{2k}}{2k\cdot n^{2k}} \right)<\frac{\abs{b_{2m}}}{2m\cdot n^{2m}}
\end{equation}
which is valid for every positive integer $m$. (see \cite[Chapter 9]{knuth2}.)
\end{proof}
\bg
\qquad For example, for every positive integer $n$, we have
\[
\ln n+\gamma+\frac{1}{2n}-\frac{1}{12n^2}<H_n<\ln n+\gamma+\frac{1}{2n}-\frac{1}{12\,n^2}+\frac{1}{120\, n^{4}}.
\]
Thus, for every $n\geq1$ we have $\gamma_n^{-}<\gamma<\gamma_n^+$ with:
\[
\gamma_n^{+}=H_n-\ln n-\frac{1}{2n}+\frac{1}{12\,n^2}\quad\text{and}\quad
\gamma_n^{-}=\gamma_n^{+}-\frac{1}{120\,n^{4}}.
\]
In Table~\ref{gamma} we find the values of these bounds for Euler's $\gamma$. It is worth noting that
formula \eqref{kn} with $m=251$ and $n=10^4$ was used by Knuth in 1962 to obtain $1271$ decimal digits of 
Euler's constant \cite{knuth}.
\bg

\begin{table}[!ht]
\begin{center}
\begin{tabular}{|c|c|c|}
\hline
 $\phantom{\Big|}n\phantom{\Big|}$ &$\gamma_n^{-}$&$\gamma_n^{+}$\\ \hline
 1 & 0.5750000000 & 0.5833333333 \\
 2 & 0.5771653194 & 0.5776861528 \\
 4 & 0.5772147535 & 0.5772473055 \\
 8 & 0.5772156500 & 0.5772176845 \\
 16 & 0.5772156647 & 0.5772157918 \\
 32 & 0.5772156649 & 0.5772156728 \\
 64 & 0.5772156649 & 0.5772156654 \\
 128 & 0.5772156649 & 0.5772156649 \\
\hline
\multicolumn{3}{c}{$\vphantom{\scriptstyle{x}}$}\\
\end{tabular}
\end{center}
\setlength{\belowcaptionskip}{10pt}
\caption{Euler's $\gamma$ belongs to the interval $(\gamma_n^{-},\gamma_n^{+})$ for each $n$.}\label{gamma}
\end{table}

\bg
\qquad Now, consider the two sequences $(c_n)_{n\geq1}$ and $(d_n)_{n\geq1}$ defined by
\[
c_n=H_n-\ln n-\gamma-\frac{1}{2n}\qquad\text{and}\qquad d_n=H_n-\ln n-\gamma
\]
For a positive integer $p$, we know according to Proposition~\ref{pr81} that $c_{pn}=\mathcal{O}\left(\frac{1}{n^2}\right)$, it follows that
the series $\sum_{n=1}^\infty c_{pn}$ is convergent. Similarly, since $d_{pn}=c_{pn}+\frac{1}{2pn}$ and the series $\sum_{n=1}^\infty(-1)^{n-1}/n$ is convergent, we conclude that $\sum_{n=1}^\infty(-1)^{n-1} d_{pn}$ is also convergent.
In what follows we aim to find asymptotic expansions, (for large $p$,) of the following sums:
\begin{align}
C_p&=\sum_{n=1}^\infty c_{pn}=\sum_{n=1}^\infty\left(H_{pn}-\ln(pn)-\gamma-\frac{1}{2pn}\right)\label{E:Cp}\\
D_p&=\sum_{n=1}^\infty (-1)^{n-1}d_{pn} =\sum_{n=1}^\infty(-1)^{n-1}\left(H_{pn}-\ln(pn)-\gamma\right)\label{E:Dp}
\end{align}

\begin{proposition}\label{pr82}
If $p$ and $m$ are positive integers and $C_p$ is defined by \eqref{E:Cp}, then
\[
C_p=-\sum_{k=1}^{m-1}\frac{b_{2k}\zeta(2k)}{2k\cdot p^{2k}}
+(-1)^m\frac{\zeta(2m)}{2m\cdot p^{2m}}\eps_{p,m},\quad\text{with $0<\eps_{p,m}<\abs{b_{2m}}$},
\]
where $\zeta$ is the well-known Riemann zeta function.
\end{proposition}
\begin{proof} Indeed, we conclude from Proposition \ref{pr81} that 
\[
H_{pn}-\ln(pn)-\gamma-\frac{1}{2pn}=-\sum_{k=1}^{m-1}\frac{b_{2k}}{2k\cdot p^{2k}}\cdot\frac{1}{n^{2k}}
+\frac{(-1)^m}{2m\cdot p^{2m}}\cdot\frac{r_{pn,m}}{n^{2m}}.
\]
with $0<r_{pn,m}\leq\abs{b_{2m}}$. It follows that
\[C_p=-\sum_{k=1}^{m-1}\frac{b_{2k}}{2k\,p^{2k}}\cdot\left(\sum_{n=1}^\infty
\frac{1}{n^{2k}}\right)+\frac{(-1)^m}{2m\cdot p^{2m}}\cdot\tilde{r}_{p,m}.
\]
where $\tilde{r}_{p,m}=\sum_{n=1}^\infty\frac{r_{pn,m}}{n^{2m}}$. 
\bg
Hence,
\[
0<\tilde{r}_{p,m}=\sum_{n=1}^\infty
\frac{r_{pn,m}}{n^{2m}}< \abs{b_{2m}} \,
\sum_{n=1}^\infty
\frac{1}{n^{2m}}=\abs{b_{2m}}\zeta(2m)
\]
and the desired conclusion follows with $\eps_{p,m}=\tilde{r}_{p,m}/\zeta(2m)$.
\end{proof}
\bg
For example, when $m=3$, we obtain
\[
\sum_{n=1}^\infty\left(H_{pn}-\ln(pn)-\gamma-\frac{1}{2pn}\right)
=-\frac{\pi^2}{72p^2}+\frac{\pi^4}{10800p^4}+\mathcal{O}\left(\frac{1}{p^6}\right).
\]
 
\bg
\qquad Similarly, in the next proposition we have the analogous result corresponding to $D_p$.\bg

\begin{proposition}\label{pr83}
If $p$ and $m$ are positive integers and $D_p$ is defined by \eqref{E:Dp}, then
\[
D_p=\frac{\ln 2}{2p}
-\sum_{k=1}^{m-1} \frac{b_{2k}\eta(2k)}{2k \cdot p^{2k}}
+(-1)^m\frac{\eta(2m)}{2m\cdot p^{2m}}\eps'_{p,m},\quad\text{with $0<\eps'_{p,m}<\abs{b_{2m}}$,}
\]
where $\eta(2k)$ is defined in Corollary~\ref{cor31}. 
\end{proposition}
\begin{proof}
Indeed, let us define $a_{n,m}$ by the formula
\[
a_{n,m}=H_n-\ln n-\gamma-\frac{1}{2n}+\sum_{k=1}^{m-1}\frac{b_{2k}}{2k\cdot n^{2k}}
\]
with empty sum equal to 0. We have shown in the proof of Proposition \ref{pr81} that
\[
(-1)^ma_{n,m}=\int_0^{1/2}\abs{B_{2m-1}(t)}g_{n,m}(t)\,dt
\]
where $g_{n,m}$ is the positive decreasing function on $[0,1/2]$ defined by
\[
g_{n,m}(t)=\sum_{j=n}^\infty\left(\frac{1}{(j+t)^{2m}}-\frac{1}{(j+1-t)^{2m}}\right).
\]
Now, for every $t\in[0,1/2]$ the sequence $(g_{np,m}(t))_{n\geq1}$ is positive and decreasing to $0$. So, Using the alternating series criterion
\cite[Theorem~7.8, and Corollary~7.9]{aman} we see that, for every $N\geq1$ and $t\in[0,1/2]$,
\[
\abs{\sum_{n=N}^\infty(-1)^{n-1}g_{np,m}(t)}\leq g_{Np,m}(t)\leq g_{Np,m}(0)=\frac{1}{(Np)^{2m}}.
\]
This proves the uniform convergence on $[0,1/2]$ of the series 
\[G_{p,m}(t)=\sum_{n=1}^\infty(-1)^{n-1}g_{np,m}(t).
\]
Consequently
\[
(-1)^m\sum_{n=1}^\infty(-1)^{n-1}a_{pn,m}=\int_0^{1/2}\abs{B_{2m-1}(t)}G_{p,m}(t)\,dt.
\]
Now using the properties of alternating series, we see that for $t\in(0,1/2)$ we have
\[
0<G_{p,m}(t)<g_{p,m}(t)<g_{p,m}(0)=\sum_{j=p}^\infty\left(\frac{1}{j^{2m}}-\frac{1}{(j+1)^{2m}}\right)=\frac{1}{p^{2m}}
\]
Thus, 
\[
\sum_{n=1}^\infty(-1)^{n-1}a_{pn,m}=\frac{(-1)^m}{p^{2m}}\rho_{p,m}
\]
with $0<\rho_{p,m}<\int_0^{1/2}\abs{B_{2m-1}(t)}\,dt$.
\bg
On the other hand we have
\begin{align*}
\sum_{n=1}^\infty(-1)^{n-1}a_{pn,m}&=D_p
-\frac{1}{2p} \sum_{n=1}^\infty\frac{(-1)^{n-1}}{n}+\sum_{k=1}^{m-1}\frac{b_{2k}}{2k\,p^{2k}}\sum_{n=1}^\infty\frac{(-1)^{n-1}}{n^{2k}}\\
&=D_p-\frac{\ln 2}{2p} +\sum_{k=1}^{m-1}\frac{b_{2k}\eta(2k)}{2k\cdot p^{2k}}.
\end{align*}
 Thus
\[
D_p=\frac{\ln 2}{2p}-\sum_{k=1}^{m-1}\frac{b_{2k}\eta(2k)}{2k\cdot p^{2k}}
+\frac{(-1)^m}{p^{2m}}\rho_{p,m}
\]
Now,  the important estimate for $\rho_{p,m}$ is the lower bound, \textit{i.e.} $\rho_{p,m}>0$. In fact, considering separately the cases $m$ odd and $m$ even, we obtain, for every nonnegative integer $m'$:
\begin{align*}
D_p&<\frac{\ln 2}{2p}-\sum_{k=1}^{2m'}\frac{b_{2k}\eta(2k)}{2k\cdot p^{2k}},\\
\noalign{\text{and}}
D_p&>\frac{\ln 2}{2p}-\sum_{k=1}^{2m'+1}\frac{b_{2k}\eta(2k)}{2k\cdot p^{2k}}.\\
\end{align*}
This yields the following more precise estimate for the error term:
\[
0<(-1)^{m}\left(D_p-\frac{\ln 2}{2p}+\sum_{k=1}^{m-1}\frac{b_{2k}\eta(2k)}{2k\,p^{2k}}
\right)<\frac{\abs{b_{2m}}\eta(2m)}{2m\cdot p^{2m}},
\]
and the desired conclusion follows.
\end{proof}
\bg

\qquad In the next lemma, we will show that there are other alternating series that can be expressed in terms of $D_p$,
This lemma will be helpful in \S~\ref{sec9}.\bg
\begin{lemma}\label{lm84}
For a positive integer $p$, we have
\begin{equation}
E_p \egdef\sum_{n=0}^\infty(-1)^{n}(H_{p(n+1)}-H_{pn}) =\ln p+\gamma-\ln\left(\frac{\pi}{2}\right)+2D_p,\label{E:Ep}
\end{equation}
where $D_p$ is the sum defined by  \eqref{E:Dp}.
\end{lemma}
\begin{proof}
Indeed
\begin{align*}
2D_p&=d_p+\sum_{n=2}^\infty(-1)^{n-1}d_{pn}+\sum_{n=1}^\infty(-1)^{n-1}d_{pn}\\
&=d_p+\sum_{n=1}^\infty(-1)^{n}d_{p(n+1)}+\sum_{n=1}^\infty(-1)^{n-1}d_{pn}\\
&=d_p+\sum_{n=1}^\infty(-1)^{n-1}(d_{pn}-d_{p(n+1)})\\
&=d_p+\sum_{n=1}^\infty(-1)^{n}(H_{p(n+1)}-H_{pn})+\sum_{n=1}^\infty(-1)^{n-1}\ln\left(\frac{n+1}{n}\right)\\
&=-\ln p-\gamma+\sum_{n=0}^\infty(-1)^{n}(H_{p(n+1)}-H_{pn})+\sum_{n=1}^\infty(-1)^{n-1}\ln\left(\frac{n+1}{n}\right)\\
\end{align*}
Using Wallis formula for $\pi$~ \cite[Formula~0.262]{grad},  we have
\begin{align*}
\sum_{n=1}^\infty(-1)^{n-1}\ln\left(\frac{n+1}{n}\right)&=
\sum_{n=1}^\infty\ln\left(\frac{2n}{2n-1}\cdot\frac{2n}{2n+1}\right)\\
&=-\ln\prod_{n=1}^\infty\left(1-\frac{1}{4n^2}\right)=\ln\left(\frac{\pi}{2}\right)\\
\end{align*}
and the desired formula follows.
\end{proof}
\bg

\section{Asymptotic expansions for certain trigonometric sums}\label{sec9}
\bn
\qquad In this section we aim to exploit the results of the previous sections to study the following trigonometric sums 
defined for a positive integer $p$ by the formul\ae:
\begin{align}
I_p&=\sum_{k=1}^{p-1}\frac{1}{\sin(k\pi /p)}=\sum_{k=1}^{p-1}\csc\left(\frac{k\pi}{p}\right)\label{E:I}\\
J_p&=\sum_{k=1}^{p-1}k\cot\left(\frac{k\pi}{p}\right)\label{E:J}
\end{align}
with empty sums interpreted as $0$. 
\bg
While there is a of favourable result \cite{kou2} concerning the sum $\sum_{k=1}^{p}\sec\left(\frac{2k\pi}{2p+1}\right)$,
 and many favourable results \cite{chen} concerning the power sums
$\sum_{k=1}^{p-1}\csc^{2n}(k\pi/p)$,  it seems that there is no known closed  form for $I_p$, and the same can be said
about the sum $J_p$. Therefore, we will look for  asymptotic expansions for these sums and will give some tight inequalities that bound $I_p$ and $J_p$.
This investigation complements the work of H. Chen  in \cite[Chapter 7.]{chen2}, and answers an open question raised there.
\bg
\qquad In the next lemma we give some equivalent forms for the trigonometric sums under consideration.
\bg
\begin{lemma}\label{lm91} For a positive integer $p$ let
\begin{alignat*}{2}
 K_p&=\sum_{k=1}^{p-1}\tan\left(\frac{k\pi}{2p}\right),
&\qquad \widetilde{K}_p&=\sum_{k=1}^{p-1}\cot\left(\frac{k\pi}{2p}\right),\\
L_p&=\sum_{k=1}^{p-1}\frac{k}{\sin(k\pi/p)},
&\qquad M_p&=\sum_{k=0}^{p-1}(2k+1)\cot\left(\frac{(2k+1)\pi}{2p}\right)
\end{alignat*}
Then,

$\ds
\begin{matrix}
\hfill i.&\hfill K_p=&\widetilde{K}_p=I_p.\hfill\label{lm911}\\
\hfill ii.&\hfill L_p=&(p/2)\,I_p.\hfill\label{lm912}\\
\hfill iii.&\hfill M_p=&(p/2)\,J_{2p}-2J_p=-p\,I_p.\hfill\label{lm913}\\
\end{matrix}
$
\end{lemma}
\begin{proof}
First, note that the change of summation variable $j\leftarrow p-j$ proves that $K_p=\widetilde{K}_p$. So,
using the trigonometric identity $\tan\theta+\cot\theta=2\csc(2\theta)$ we obtain
\begin{align*}
2K_p&=K_p+\widetilde{K}_p=\sum_{k=1}^{p-1}\left(\tan\left(\frac{k\pi}{2p}\right)+\cot\left(\frac{k\pi}{2p}\right)\right)\\
&=2\sum_{k=1}^{p-1}\csc\left(\frac{k\pi}{p}\right)=2I_p
\end{align*}
This proves $(i)$.\bg
Similarly, $(ii)$ follows from the change of summation variable $j\leftarrow p-j$ in $L_p$:
\[
L_p=\sum_{j=1}^{p-1}\frac{p-k}{\sin(k\pi/p)}=pI_p-L_p
\]
Also,
\begin{align*}
M_p&=\sum_{\substack{1\leq k<2p\\ k \text{ odd}
}} k\cot\left(\frac{k\pi}{2p}\right)=\sum_{k=1}^{2p-1} k\cot\left(\frac{k\pi}{2p}\right)- \sum_{\substack{1\leq k<2p\\ k \text{ even}
}} k\cot\left(\frac{k\pi}{2p}\right)\\
&=\sum_{k=1}^{2p-1} k\cot\left(\frac{k\pi}{2p}\right)-2 \sum_{k=1}^{p-1} k\cot\left(\frac{k\pi}{p}\right)=J_{2p}-2J_p.
\end{align*}
But
\begin{align*}
J_{2p}&=\sum_{k=1}^{p-1} k\cot\left(\frac{k\pi}{2p}\right)+\sum_{k=p+1}^{2p-1} k\cot\left(\frac{k\pi}{2p}\right)\\
&=\sum_{k=1}^{p-1} k\cot\left(\frac{k\pi}{2p}\right)-\sum_{k=1}^{p-1} (2p-k)\cot\left(\frac{k\pi}{2p}\right)\\
&=2\sum_{k=1}^{p-1} k\cot\left(\frac{k\pi}{2p}\right)-2p\widetilde{K}_p
\end{align*}
Thus, using $(i)$ and the trigonometric identity $\cot(\theta/2)-\cot\theta=\csc\theta$ we obtain
\begin{align*}
M_p&=J_{2p}-2J_p=2\sum_{k=1}^{p-1} k\left(\cot\left(\frac{k\pi}{2p}\right)-\cot\left(\frac{k\pi}{p}\right)\right)
-2pI_p\\
&=2\sum_{k=1}^{p-1}k\csc\left(\frac{k\pi}{p}\right)-2pI_p=2L_p-2pI_p=-pI_p
\end{align*}
This concludes the proof of $(iii)$.
\end{proof}

\bg

\begin{proposition}\label{pr92}
For $p\geq2$, let $I_p$ be the sum of cosecants defined by the \eqref{E:I}. Then
\begin{align*}
I_p&=-\frac{2\ln 2}{\pi}+\frac{2p}{\pi}E_p,\\
&=-\frac{2\ln 2}{\pi}+\frac{2p}{\pi}\left(\ln p+\gamma-\ln(\pi/2)\right)+\frac{4p}{\pi}D_p,
\end{align*}
where $D_p$ and $E_p$ are defined by formul\ae~ \eqref{E:Dp} and \eqref{E:Ep} respectively.
\end{proposition}
\begin{proof}
Indeed, our starting point will be the ``simple fractions''  expansion \eqref{E:secf} of the cosecant function:
\[
\frac{\pi}{\sin(\pi\alpha)}=\sum_{n\in\ent}\frac{(-1)^n}{\alpha-n}=\frac{1}{\alpha}+\sum_{n=1}^\infty(-1)^n\left(\frac{1}{\alpha-n}+
\frac{1}{\alpha+n}\right)
\]
which is valid for $\alpha\in\comp\setminus\ent$. Using this formula with $\alpha=k/p$ for  $k=1,2,\ldots,p-1$ and adding, we conclude that
\begin{align*}
\frac{\pi}{p} I_p&=\sum_{k=1}^{p-1}\frac{1}{k}+\sum_{n=1}^\infty(-1)^n\sum_{k=1}^{p-1}\left(\frac{1}{k-np}+
\frac{1}{k+n p}\right)\\
&=\sum_{k=1}^{p-1}\frac{1}{k}+ \sum_{n=1}^\infty(-1)^n\left(-\sum_{j=p(n-1)+1}^{pn-1}\frac{1}{j}+
\sum_{j=pn+1}^{p(n+1)-1}\frac{1}{j}\right),
\end{align*}
\bg
and this result can be expressed in terms of the Harmonic numbers as follows
\begin{align*}
\frac{\pi}{p} I_p&=H_{p-1}+ \sum_{n=1}^\infty(-1)^n\left(- H_{pn-1}+H_{p(n-1)}+H_{p(n+1)-1}-H_{pn}
\right)\\
&=H_{p-1}+ \sum_{n=1}^\infty(-1)^n\left(H_{p(n+1)}-2H_{pn}+H_{p(n-1)}\right)
+\frac{1}{p}\sum_{n=1}^\infty(-1)^n\left(\frac{1}{n}-\frac{1}{n+1}
\right)\\
&=H_{p-1}+\sum_{n=1}^\infty(-1)^n\left(H_{p(n+1)}-2H_{pn}+H_{p(n-1)}\right)
+\frac{1}{p}\left(\sum_{n=1}^\infty\frac{(-1)^n}{n}+\sum_{n=2}^\infty\frac{(-1)^n}{n}\right)\\
&=H_{p}+\sum_{n=1}^\infty(-1)^n\left(H_{p(n+1)}-2H_{pn}+H_{p(n-1)}\right)
-\frac{2}{p}\sum_{n=1}^\infty(-1)^{n-1}\frac{1}{n}\\
&=H_{p}-\frac{2\ln 2}{p}+\sum_{n=1}^\infty(-1)^n\left(H_{p(n+1)}-2H_{pn}+H_{p(n-1)}\right).
\end{align*}
Thus
\begin{align*}
\frac{\pi}{p} I_p+\frac{2\ln 2}{p}&=
H_{p}+\sum_{n=1}^\infty(-1)^n\left(H_{p(n+1)}-H_{pn} \right)
+\sum_{n=1}^\infty(-1)^n\left(H_{p(n-1)}-H_{pn}\right)\\
&=\sum_{n=0}^\infty(-1)^n\left(H_{p(n+1)}-H_{pn}\right)
+\sum_{n=1}^\infty(-1)^n\left(H_{p(n-1)}-H_{pn}\right)\\
&=E_p+E_p=2E_p,
\end{align*}
and the desired formula follows according to Lemma~\ref{lm84}.
\end{proof}

\qquad Combining Proposition~\ref{pr92} and Proposition~\ref{pr83}, we obtain:
\begin{proposition}\label{pr93} 
For $p\geq2$ and $m\geq 1$, we have
\[
 \pi I_p= 2p\ln p+2(\gamma-\ln( \pi/2))p -\sum_{k=1}^{m-1} \frac{2b_{2k}\eta(2k)}{ k \cdot p^{2k-1}}
+(-1)^m\frac{2\eta(2m)}{ m\cdot p^{2m-1}}\eps'_{p,m}
\]
with $\ds 0<\eps'_{p,m}<\abs{b_{2m}}$.
\end{proposition}
\qquad Using the values of the $\eta(2k)$'s from  Corollary \ref{cor31}, and considering separately the cases $m$ even and $m$ odd we obtain
the following corollary.
 
\begin{corollary}\label{cor94}
For every positive integer $p$ and every nonnegative integer $n$, the sum of cosecants $I_p$ defined by \eqref{E:I} satisfies the following
inequalities:
\begin{align*}
I_p&<\frac{2p}{\pi}(\ln p+\gamma-\ln(\pi/2))+
\sum_{k=1}^{2n}(-1)^{k}\frac{(2^{2k}-2)b_{2k}^2}{k\cdot(2k)!}\left(\frac{\pi}{p}\right)^{2k-1},\\
\noalign{\text{and}}
I_p&>\frac{2p}{\pi}(\ln p+\gamma-\ln(\pi/2))+
\sum_{k=1}^{2n+1}(-1)^{k}\frac{(2^{2k}-2)b_{2k}^2}{k\cdot(2k)!}\left(\frac{\pi}{p}\right)^{2k-1}.
\end{align*}
\end{corollary}
\bg
\qquad As an example, for $n=0$ we obtain the following inequality, valid for every $p\geq1$ :
\[
\frac{2p}{\pi}(\ln p+\gamma-\ln(\pi/2))-\frac{\pi}{36p}
<I_p<\frac{2p}{\pi}(\ln p+\gamma-\ln(\pi/2)).
\]
This answers positively the open problem proposed in \cite[Section 7.4]{chen2}.
\bg
\begin{remark}
The asymptotic expansion of $I_p$ was proposed as an exercise in \cite[Exercise~13, p.~460]{hen}, and 
it was attributed to P. Waldvogel, but the result there is less precise than Corollary~\ref{cor94}
\end{remark}
\bg
\qquad Now we turn our attention to the other trigonometric sum $J_p$.
Next we find an analogous result to Proposition~\ref{pr92}.
\bg
\begin{proposition}\label{pr95}
For a positive integer $p$, 
let $J_p$ be the sum of cotangents defined by the \eqref{E:J}. Then
\[
\pi J_p= -p^2\ln p+(\ln(2\pi)-\gamma)p^2 -p+2p^2C_p
\]
where $C_p$ is defined by the formula \eqref{E:Cp}.
\end{proposition}
\begin{proof}
Recall that $c_{n}=H_n-\ln n-\gamma-\frac{1}{2n}$ satisfies
 $c_n=\mathcal{O}(1/n^2)$. Thus, both series 
\[
C_p=\sum_{n=1}^\infty c_{pn}\quad\text{ and }\quad \widetilde{C}_p=\sum_{n=1}^\infty(-1)^{n-1} c_{pn}
\] are convergent. Further, we note that $\widetilde{C}_p=D_p-\frac{\ln 2}{2p}$ where $D_p$ is defined by \eqref{E:Dp}.

According to Remark~\ref{rm98} we have
\begin{equation}\label{E:pr991}
\widetilde{C}_p =\frac{\ln(\pi/2)-\gamma-\ln p}{2}+\frac{\pi}{4p}I_p.
\end{equation}
Now, noting that
\begin{align*}
C_p&=\sum_{\substack{n\geq1\\
                                     n\,\text{odd}}}c_{pn}
         +\sum_{\substack{n\geq1\\
                                     n\,\text{even}}}c_{pn}
=\sum_{\substack{n\geq1\\
                             n\,\text{odd}}}c_{pn}+\sum_{n=1}^\infty c_{2pn}\\
\widetilde{C}_p&=\sum_{\substack{n\geq1\\
                                      n\,\text{odd}}}c_{pn}
          -\sum_{\substack{n\geq1\\
                                      n\,\text{even}}}c_{pn}
=\sum_{\substack{n\geq1\\
                             n\,\text{odd}}}c_{pn}-\sum_{n=1}^\infty c_{2pn}
\end{align*}
we conclude that $C_p-\widetilde{C}_p=2C_{2p}$, or equivalently
\begin{equation}\label{E:pr992}
C_p-2C_{2p}=\widetilde{C}_p
\end{equation}
On the other hand, for a positive integer $p$ let us define $F_p$ by
\begin{equation}\label{E:Fp}
F_p=\frac{\ln p+\gamma-\ln(2\pi)}{2}+\frac{1}{2p}+\frac{\pi}{2p^2}J_p.
\end{equation}
It is easy to check, using Lemma~\ref{lm91} $(iii)$, that
\begin{align}\label{E:pr994}
F_p-2F_{2p}&=\frac{\ln(\pi/2)-\ln p-\gamma}{2} -\frac{\pi}{4p^2}(J_{2p}-2J_p)\notag\\
&=\frac{\ln(\pi/2)-\ln p-\gamma}{2} +\frac{\pi}{4p}I_p
\end{align}
We conclude from \eqref{E:pr992} and \eqref{E:pr994} that $C_p-2C_{2p}=F_p-2F_{2p}$, or equivalently
\[C_p-F_p=2(C_{2p}-F_{2p}).\]
 Hence, 
\begin{equation}\label{E:pr995}
\forall\,m\geq1,\qquad C_p-F_p=2^m(C_{2^mp}-F_{2^mp})
\end{equation}
Now, using  Proposition~\ref{pr81} to replace $H_p$ in Proposition~\ref{pr72},  we obtain
\begin{align*}
\frac{\pi}{p^2}J_p&=\ln(2\pi)-H_p-\frac{1}{2p}+\mathcal{O}\left(\frac{1}{p^2}\right)\\
&=\ln(2\pi)-\ln p-\gamma-\frac{1}{p} +\mathcal{O}\left(\frac{1}{p^2}\right)
\end{align*}
Thus $F_p=\mathcal{O}\left(\frac{1}{p^2}\right)$. Similarly, from the fact that $c_n=\mathcal{O}\left(\frac{1}{n^2}\right)$
we conclude also that $C_p=\mathcal{O}\left(\frac{1}{p^2}\right)$. Consequently, there exists a constant $\kappa$  such that, for large values of $p$
we have $\abs{C_p-F_p}\leq \kappa/p^2$. So, from \eqref{E:pr995}, we see that for large values of $m$ we have
\[
\abs{C_p-F_p}\leq\frac{\kappa}{2^mp^2}
\]
and letting $m$ tend to $+\infty$ we obtain $C_p=F_p$, which is equivalent to the announced result.
\end{proof}
\bg
\qquad Combining Proposition~\ref{pr95} and Proposition~\ref{pr82}, we obtain:
\begin{proposition}\label{pr96} 
For $p\geq2$ and $m\geq 1$, we have
\[
\pi J_p= -p^2\ln p+(\ln(2\pi)-\gamma)p^2 -p -\sum_{k=1}^{m-1}\frac{b_{2k}\zeta(2k)}{ k\cdot p^{2k-2}}
+(-1)^m\frac{\zeta(2m)}{m\cdot p^{2m-2}}\eps_{p,m},
\]
with $0<\eps_{p,m}<\abs{b_{2m}}$, where $\zeta$ is the well-known Riemann zeta function.
\end{proposition}
\qquad Using the values of the $\zeta(2k)$'s from  Corollary \ref{cor31}, and considering separately
 the cases $m$ even and $m$ odd we obtain
the next corollary.
 
\begin{corollary}\label{cor97}
For every positive integer $p$ and every nonnegative integer $n$, the sum of cotangents $J_p$ defined by \eqref{E:J} satisfies the following
inequalities:
\begin{align*}
J_p&< \frac{1}{\pi}\left(-p^2\ln p+(\ln(2\pi)-\gamma)p^2 -p\right) 
+2\pi\sum_{k=1}^{2n}(-1)^k\frac{b_{2k}^2}{ k\cdot(2k)!} \left(\frac{2\pi}{ p}\right)^{2k-2},\\
\noalign{\text{and}}
J_p&> \frac{1}{\pi}\left(-p^2\ln p+(\ln(2\pi)-\gamma)p^2 -p\right) 
+2\pi\sum_{k=1}^{2n+1}(-1)^k\frac{b_{2k}^2}{ k\cdot(2k)!} \left(\frac{2\pi}{ p}\right)^{2k-2}.
\end{align*}
\end{corollary}
\bg
\qquad As an example, for $n=0$ we obtain the following double inequality, which is valid for  $p\geq1$ :
\[
0 < \frac{1}{\pi}\left(-p^2\ln p+(\ln(2\pi)-\gamma)p^2 -p\right)-J_p<\frac{\pi}{36}
\]
\bg
\begin{remark}
It is not clear that one can prove Proposition~\ref{pr96} by combining Propositions  \ref{pr72} and \ref{pr81} directly.
\end{remark}

\begin{remark} \label{rm98} Note that we have proved the following results. For a postive integer $p$:
\begin{align*}
\sum_{n=1}^\infty(-1)^{n-1}(H_{pn}-\ln(pn)-\gamma)&=\frac{\ln(\pi/2)-\gamma-\ln p}{2}+\frac{\ln 2}{2p}
+\frac{\pi}{4p}\sum_{k=1}^{p-1}\csc\left(\frac{k \pi}{p}\right).\\
\sum_{n=0}^\infty(-1)^{n}(H_{p(n+1)}-H_{pn})&=\frac{\ln 2}{p}
+\frac{\pi}{2p}\sum_{k=1}^{p-1}\csc\left(\frac{k \pi}{p}\right).\\
\sum_{n=1}^\infty\left(H_{pn}-\ln(pn)-\gamma-\frac{1}{2pn}\right)&=
\frac{\ln p+\gamma-\ln(2\pi)}{2}+\frac{1}{2p}+\frac{\pi}{2p^2}\sum_{k=1}^{p-1}k\cot\left(\frac{k \pi}{p}\right).
\end{align*}
These results are to be compared with those in \cite{kou}, see also \cite{kou1}.
\end{remark}
\bg
\section{Endnotes}\label{final}
\bn
\qquad Bernoulli numbers appeared in the work  
Jacob Bernoulli (1654-1705, Basel, Switzerland) in connection with evaluating the sums of the form
$\sum_{i=1}^ni^k$. They appear in Bernoulli's most original work \textsl{Ars Conjectandi},
 (``The Art of Conjecturing'',) published by his nephew
in Basel in 1713, eight years after his death. They also appear independantly in the work of
the Japanese Mathematician Seki Takakazu (Seki K\=owa) (1642-1708). Bernoulli numbers
are related to Fermat's Last Theorem by Kummer's theorem \cite{kum}, this opened a very large and furctuous
field of investigation (see \cite[Chapter~15]{Ire}.)

\qquad The Euler-Maclaurin's Summation Formula was developed independantly by Leonhard Euler (1736)
and Colin MacLaurin (1742). The aim was to provide a powerful tool for the computation of
several sums like the harmonic numbers. 

\qquad Finally, a thorough bibliography on Bernoulli numbers and their applications that contains more than $3000$ entries can be found on the 
web \cite{dil3}.



\begin{thebibliography}{9}
\setlength{\itemsep}{5pt}

\bibitem{abr}
Abramowitz,~M. and Stegan,~I.~A.,  
\emph{Handbook of Mathematical Functions, with Formulas, Graphs, and
Mathematical Tables}, Dover Books on Mathematics,
 Dover Publication, Inc., New York, (1972).


\bibitem{ahl}
Ahlfors,~L.~V.,  
\emph{Complex Analysis},   McGraw-Hill, Inc. (1979).

\bibitem{aman}
Amann,~H. and Escher,~J.,  
\emph{Analysis~ I}, Birkh\"auser Verlag, Basel -Boston-Berlin. (2005).

\bibitem{ath}
Athreya,~K.~B.  and Lahiri,~S.~N.,
\emph{Measure Theory and Probability Theory}, Springer Science+Business Media, LLC. (2006).

\bibitem{ber}
Bergmann,~H.,
\emph{Eine explizite Darstellung der Bernoullischen Zahlen}, Math. Nachr., \textbf{34}, (1967), pp.377-378.

\bibitem{chen}
Chen,~H.,
\emph{On some trigonometric power sums},
Internat. J. Math. Math. Sci, \textbf {30}, (2002), 185--191.

\bibitem{chen2}
\bysame,
\emph{Excursions in Classical Analysis},
Mathematical Association of America, Inc. (2010).

\bibitem{dee}
Deeba,~E. Y. and Rodrigues,~D. M.,  
\emph{Stirling's Series and Bernoulli Numbers}, 
The American Mathematical Monthly, \textbf{98}, 5  (1991), pp.423--426.

\bibitem{dil}
Dilcher,~K.,
\emph{Asymptotic Behaviour of Bernoulli, Euler, and Generalized Bernoulli Polynomials},
Journal of Approximation Theory, \textbf {49}, (1987), 321--330.

\bibitem{dil2}
\bysame,
\emph{Zeros of Bernoulli, generalized Bernoulli and Euler polynomials},
Mem. Amer. Math. Soc., Number 386, (1988).

\bibitem{dil3}
Dilcher,~K. and Slavutskii,~I. Sh.,
\emph{A Bibliography of Bernoulli Numbers (1713-2007)}.

[ONLINE : \texttt{http://http://www.mathstat.dal.ca/$\sim$dilcher/bernoulli.html}].

\bibitem{gld}
Gould.~H. W.,  
\emph{Explicit formulas for Bernoulli numbers}, The American Mathematical Monthly, \textbf{79}, 1  
(1972), pp.44--51.

\bibitem{grad}
Gradshteyn,~I. and Ryzhik,~I.,  
\emph{Tables of Integrals, Series and Products}, 7th ed., Academic Press, (2007).

\bibitem{gra}
Grafakos,~L.,
\emph{Classical Fourier Analysis, Second edition.}, Graduate Texts in Mathemetics, 
 Springer Science+Business Media, LLC. (2008).

\bibitem{knuth2}
Graham,~R. L., Knuth,~D. E., and Patashnik, O.
\emph{Concrete Mathematics : a foundation for computer science}, 2nd ed.
 Addison-Wesley Publishing Company, Inc, (1994).

\bibitem{har}
Hardy,~G.~H. and Wright,~E.~M.,  
\emph{An introduction to the theory of numbers}, 6th ed., Oxford University Press, (2007).

\bibitem{hen}
Henrici,~P.,  
\emph{Applied and Computational Complex Analysis, Vol. 2}, John Wiley \& Sons, New York, (1977).

\bibitem{ink}
Inkeri,~K., 
\emph{The real roots of Bernoulli polynomials}, Ann. Univ.
Turku. Ser. A \textbf{I37} (1959), 20pp.

\bibitem{Ire}
Ireland,~K. and Rosen,~M., 
\emph{A Classical Introduction to Modern Number Theory}, 2nd ed., Springer-Verlag, New York, Inc, (1990).

\bibitem{katz}
Katznelson,~Y.,  
\emph{Introduction To Harmonic Analysis}, 3rd ed., Cambridge University Press, (2004).

\bibitem{knuth}
Knuth,~D. E.,
\emph{Euler's constant to 1271 places}, Math. Comput., \textbf{16}, (1962), pp. 275--281.


\bibitem{kou}
Kouba,~O.,  
\emph{The sum of certain series related to harmonic numbers}, Octogon Mathematical Magazine, \textbf{19}, No. 1 (2011), pp.3--18.

[ONLINE : \texttt{http://www.uni-miskolc.hu/$\sim$matsefi/Octogon/}].

\bibitem{kou1}
\bysame,  
\emph{Proposed Problem 11499}, The American Mathematical Monthly, \textbf{117}, 7  (2010), p.371.

\bibitem{kou2}
Kouba,~O., and Andreescu,~T.,  
\emph{Solution to Problem U207}, Mathematical Reflections, Issue 5,  (2011), p.16.

[ONLINE : \texttt{http://www.awesomemath/home}].


\bibitem{kum}
Kummer~E. E.,\emph{Allgemeiner
 Beweis des Fermat'schen Satzes, dass die Gleichung $x^\lambda + y^\lambda = z^\lambda$ durch ganze
 Zahlen unl\"osbar ist, f\"ur alle diejenigen Potenz-Exponenten $\lambda$, welche 
ungerade Primzahlen sind und in den Z\"ahlern der ersten $(\lambda-3)/2$ 
Bernoulli'schen Zahlen als Factoren nicht vorkommen}, J. Reine Angew. Math. \textbf{40}  (1850), pp.131--138.

\bibitem{leh}
Lehmer,~D. H.,  
\emph{On the maxima and minima of Bernoulli Polynomials}, 
The American Mathematical Monthly, \textbf{47}, 8  (1940). pp, 533--358.

\bibitem{nam}
Namias,~V.,  
\emph{A simple Derivation of Stirling's Asymptotic Series}, 
The American Mathematical Monthly, \textbf{93}, 1  (1986). pp, 25--29.


\end{thebibliography}
\end{document}